\documentclass[11pt,a4paper]{amsart}
\usepackage[margin=25mm]{geometry}
\usepackage{amsmath,amssymb,amsthm,mathtools,mathrsfs,bm,cases}
\usepackage{microtype}
\usepackage[hidelinks]{hyperref}
\usepackage{enumitem}
\usepackage{tikz}
\usetikzlibrary{arrows.meta,positioning,calc,patterns}
\definecolor{proofblue}{RGB}{35,91,145}
\definecolor{wallorange}{RGB}{187,105,24}
\setlist{itemsep=2pt,topsep=4pt,parsep=0pt}
\hypersetup{pdfauthor={Kailun Chen and Xincheng Zhang},pdftitle={Half-space polynuclear growth: a multipoint Fredholm--Pfaffian formula}}
\allowdisplaybreaks[2]
\numberwithin{equation}{section}
\newtheorem{theorem}{Theorem}[section]
\newtheorem{proposition}[theorem]{Proposition}
\newtheorem{lemma}[theorem]{Lemma}
\newtheorem{corollary}[theorem]{Corollary}
\theoremstyle{definition}
\newtheorem{definition}[theorem]{Definition}
\theoremstyle{remark}

\newcommand{\ind}{\mathbf1}
\newcommand{\Pf}{\operatorname{Pf}}

\newcommand{\diag}{\operatorname{diag}}
\newcommand{\tr}{\operatorname{tr}}

\newcommand{\HH}{\mathcal H}

\newcommand{\scD}{\mathscr D}
\newcommand{\cL}{\mathcal L}
\newcommand{\cX}{\widehat{\mathcal X}}
\newcommand{\cA}{\mathcal A}

\newcommand{\Fhat}{\widehat F}

\title{Half-space polynuclear growth:\\a multipoint Fredholm--Pfaffian formula}
\author{Kailun Chen \and Xincheng Zhang}
\date{}

\tikzset{
 png figure/.style={font=\small,>=Stealth,line cap=round,line join=round},
 png axis/.style={->,black!75,line width=.45pt},
 png profile/.style={black!85,line width=.85pt},
 png path/.style={proofblue,line width=.95pt},
 png wall/.style={wallorange,line width=1.05pt},
 png guide/.style={black!45,densely dashed,line width=.45pt},
 png endpoint/.style={proofblue,densely dashed,line width=.75pt},
 png edge/.style={wallorange,line width=1pt},
 png motion/.style={->,black!85,line width=.75pt},
 png region/.style={fill=proofblue!9,draw=none},
 png second region/.style={fill=wallorange!10,draw=none},
 png dot/.style={circle,fill=black!85,inner sep=1.5pt},
 png open/.style={circle,draw=black!85,fill=white,line width=.6pt,inner sep=1.3pt}
}

\begin{document}

\vspace{-1.2em}
\begin{abstract}
We derive a Fredholm-Pfaffian formula for the one-time multi-point transition probabilities of half-space polynuclear growth with any nonnegative boundary nucleation rate and general deterministic upper semi-continuous initial profiles. We prove the formula by showing that it satisfies an integral Kolmogorov equation with a unique bounded solution. We also discuss the formal $1:2:3$ scaling limit of the kernel and its relation to the half-space KPZ fixed point.
\end{abstract}

\maketitle

\section{Introduction}\label{sec:model}
\subsection{Overview}\label{subsec:literature}
The Kardar--Parisi--Zhang (KPZ) universality class~\cite{KPZ86} describes common large-scale fluctuation behavior in a broad family of random growth models. In one spatial dimension, the
characteristic KPZ scales are $t^{1/3}$ for height fluctuations around the deterministic limit shape and $t^{2/3}$ for spatial correlation lengths. The limiting laws also depend on the initial geometry. Exactly solvable models have been essential for identifying these laws and understanding how KPZ fluctuations arise from microscopic growth rules.

Polynuclear growth (PNG) is one of the main exactly solvable models in this class. Full-space PNG has been studied for droplet, flat, and stationary initial data. For droplet initial condition, the one-point height fluctuations converge to the GUE Tracy--Widom law~\cite{PS00}. With additional nucleation sources at the two moving droplet edges, Baik and Rains~\cite{BR00} obtained source-dependent one-point limiting laws, including GUE, GOE$^2$, $F_0$, and Gaussian limits. For one-time spatial multi-point fluctuations, Pr\"ahofer and Spohn~\cite{PS02} obtained the Airy$_2$ process from the droplet PNG, and Johansson~\cite{Joh03} proved functional convergence to this process for discrete droplet PNG. Imamura and Sasamoto~\cite{IS04} derived Fredholm determinant representations and spatial crossover kernels for discrete droplet PNG with edge sources. Johansson and
Rahman~\cite{JR21} obtained multi-time distributions and their
KPZ scaling limits for discrete droplet PNG. For deformed PNG,
Aggarwal, Borodin, and Wheeler~\cite{ABW21} proved one-point
GUE Tracy--Widom asymptotics. For flat initial conditions, the one-point height fluctuations converge to the GOE Tracy--Widom law~\cite{PS00}. Ferrari~\cite{Fer04} proved that the rescaled multilayer heights at a fixed spatial position converge to the GOE edge Pfaffian point process. Borodin, Ferrari, and Sasamoto~\cite{BFS08} obtained spatial multi-point height distributions and their Airy$_1$ limit, more generally along space-like paths. For stationary initial data, the one-point
height fluctuations along the characteristic direction converge
to the Baik--Rains distribution $F_0$~\cite{PS00}. Pr\"ahofer
and Spohn~\cite{PS04} computed the exact scaling function of
the stationary two-point space-time correlation.

A major advance beyond these special initial conditions was made
by Matetski, Quastel, and Remenik~\cite{MQR}, who obtained an
exact Fredholm determinant formula for the one-time spatial
multi-point transition probabilities of full-space PNG from
arbitrary deterministic upper semi-continuous initial profiles.
The initial profile enters the kernel through hitting operators
for continuous-time symmetric simple random walks. Their formula
provides a unified description of the finite-time transition laws,
extending exact solvability from selected initial geometries to the  general initial data.

The half-space geometry leads to different fluctuation laws
near the wall. Sasamoto and Imamura~\cite{SI04} studied
half-space droplet PNG with nucleation at the origin, in both
the discrete model and its continuous limit. At boundary
nucleation rates 0 and 1, they
derived exact one-time spatial multipoint formulas and obtained
spatial scaling limits describing symplectic-to-unitary and
orthogonal-to-unitary soft-edge transitions, respectively. These spatial limits connect the
previously known GSE and GOE Tracy--Widom laws at the boundary
to the Airy$_2$ process in the bulk.

In this paper, we extend the exact half-space transition formulas beyond droplet initial data and the special boundary nucleation rates. We derive a Fredholm--Pfaffian formula for the
one-time multi-point transition probabilities of half-space PNG at every nonnegative boundary nucleation rate and for arbitrary
deterministic upper semi-continuous initial profiles.

\subsection{The half-space PNG model}
\label{subsec:model}
\paragraph{State space and dynamics.} The half-space polynuclear growth (PNG) model is a Markov process that describes the stochastic growth of crystal surfaces in a supersaturated atmosphere. The state space of the process is
\begin{equation}\label{eqn:state-space}
 \cX=
 \left\{
 h:[0,\infty)\longrightarrow\mathbb Z\cup\{-\infty\}:
 \begin{array}{l}
 h\text{ is upper semicontinuous},\\
 \text{for every }B>0\text{ there is a finite partition}\\
 0=y_0<y_1<\cdots<y_N=B\text{ such that }h\\
 \text{is constant on each }(y_{j-1},y_j),\quad 1\le j\le N
 \end{array}
 \right\}.
\end{equation} 
We equip $\cX$ with the local Hausdorff topology of the hypographs
\[
 \operatorname{hypo}(h)
 =\{(x,u)\in[0,\infty)\times\mathbb R:u\le h(x)\}.
\]
The value $-\infty$ represents the absence of a surface. Wherever the
one-sided heights are finite, jumps may have arbitrary integer size.
For $z>0$, a point satisfying
$h(z)>\max\{h(z^-),h(z^+)\}$ is an \emph{isolated upper peak}.
At the origin, we also allow \emph{excess wall height}, meaning
$h(0)>h(0^+)$. For the step-coordinate calculations below, write
\[
 \cX_{\mathrm{fin}}
 :=\{h\in\cX:h(x)\in\mathbb Z\text{ for every }x\ge0\}.
\]

We first describe the dynamics in words. For a finite-valued profile, up steps move to the left and down steps to the right deterministically, both at speed one. A down step of size $a$ on the left and an up step of size $b$ on the right approach one another. At their collision, their common portion annihilates: if $a=b$, both steps disappear; otherwise a jump of size $b-a$ remains and moves according to its sign. On the surface of the height function, there will be random nucleation, which creates a peak one unit above the current profile. Its edges spread until they meet other steps.

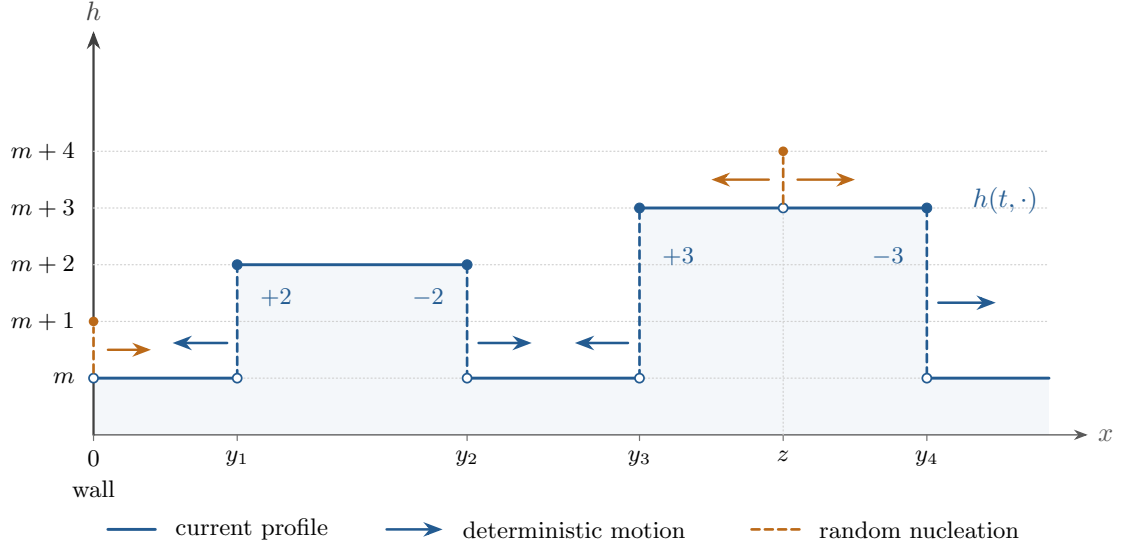
\begin{figure}[t]
\centering
\begin{tikzpicture}[x=1.9cm,y=.75cm,font=\small,>=Stealth,
  line cap=round,line join=round]
\tikzset{
 cfg profile/.style={proofblue,line width=1.1pt},
 cfg motion/.style={proofblue,->,line width=.95pt},
 cfg birth/.style={wallorange,densely dashed,line width=1pt},
 cfg side/.style={proofblue,densely dashed,line width=1pt},
 cfg birth motion/.style={wallorange,->,line width=.9pt},
 cfg guide/.style={black!18,densely dotted,line width=.45pt},
 cfg closed/.style={circle,draw=proofblue,fill=proofblue,inner sep=1.3pt},
 cfg open/.style={circle,draw=proofblue,fill=white,line width=.65pt,inner sep=1.2pt}
}
\fill[proofblue!5] (0,0)--(6.65,0)--(6.65,1)--(5.8,1)
 --(5.8,4)--(3.8,4)--(3.8,1)--(2.6,1)--(2.6,3)
 --(1,3)--(1,1)--(0,1)--cycle;
\draw[black!65,->,line width=.5pt] (0,0)--(6.92,0) node[right] {$x$};
\draw[black!75,line width=.85pt] (0,0)-- (0,1);
\draw[black!75,->,line width=.85pt] (0,2)--(0,7.12) node[above] {$h$};
\foreach \yy/\lab in {1/m,2/{m+1},3/{m+2},4/{m+3},5/{m+4}}{
 \draw[cfg guide] (0,\yy)--(6.65,\yy);
 \node[left,font=\footnotesize] at (-.06,\yy) {$\lab$};
}
\draw[cfg profile] (0,1)--(1,1);
\draw[cfg side] (1,1)--(1,3);

\draw[cfg profile](1,3)--(2.6,3);
\draw[cfg side] (2.6,3)--(2.6,1);
\draw[cfg profile](2.6,1)
 --(3.8,1);
 \draw[cfg side] (3.8,1)--(3.8,4);
 \draw[cfg profile](3.8,4)--(5.8,4);
 \draw[cfg side] (5.8,4)--(5.8,1);
 \draw[cfg profile](5.8,1)--(6.65,1);

\foreach \p in {(1,3),(2.6,3),(3.8,4),(5.8,4)}{
 \node[cfg closed] at \p {};
}
\foreach \p in {(0,1),(1,1),(2.6,1),(3.8,1),(4.8,4),(5.8,1)}{
 \node[cfg open] at \p {};
}
\foreach \xx/\lab in {0/0,1/{y_1},2.6/{y_2},3.8/{y_3},4.8/z,5.8/{y_4}}{
 \draw[black!55,line width=.45pt] (\xx,0)--(\xx,-.1);
 \node[below,font=\footnotesize] at (\xx,-.1) {$\lab$};
}
\node[below,font=\footnotesize] at (0,-.62) {wall};
\node[proofblue,anchor=west] at (6.05,4.13) {$h(t,\cdot)$};

\node[proofblue,anchor=west,font=\footnotesize] at (1.09,2.43) {$+2$};
\node[proofblue,anchor=east,font=\footnotesize] at (2.51,2.43) {$-2$};
\node[proofblue,anchor=west,font=\footnotesize] at (3.89,3.15) {$+3$};
\node[proofblue,anchor=east,font=\footnotesize] at (5.71,3.15) {$-3$};
\draw[cfg motion] (.93,1.62)--(.55,1.62)
 node[midway,above,font=\footnotesize] {};
\draw[cfg motion] (2.68,1.62)--(3.05,1.62)
 node[midway,above,font=\footnotesize] {};
\draw[cfg motion] (3.72,1.62)--(3.35,1.62)
 node[midway,above,font=\footnotesize] {};
\draw[cfg motion] (5.88,2.33)--(6.28,2.33)
 node[midway,above,font=\footnotesize] {};



\fill[wallorange] (0,2) circle (1.8pt);
\draw[cfg birth] (0,1.1)--(0,1.89);
\draw[cfg birth motion] (0.1,1.5)--(0.4,1.5);

\draw[cfg guide] (4.8,0)--(4.8,3.95);
\draw[cfg birth] (4.8,4.08)--(4.8,4.89);
\fill[wallorange] (4.8,5) circle (1.8pt);
\draw[cfg birth motion] (4.9,4.5)--(5.3,4.5);
\draw[cfg birth motion] (4.7,4.5)--(4.3,4.5);

\begin{scope}[yshift=-1.25cm,font=\footnotesize]
 \draw[cfg profile] (.1,0)--(.43,0);
 \node[anchor=west] at (.5,0) {current profile};
 \draw[cfg motion] (2.04,0)--(2.43,0);
 \node[anchor=west] at (2.5,0) {deterministic motion};
 \draw[cfg birth] (4.58,0)--(4.92,0);
 \node[anchor=west] at (4.98,0) {random nucleation};
\end{scope}

\end{tikzpicture}
\caption{
A representative half-space PNG configuration and its evolution
rules, allowing jumps larger than one.}
\label{fig:halfspace-png-model}
\end{figure}
Figure~\ref{fig:halfspace-png-model} shows a representative configuration
and the evolution rules just described.

Now we state it more precisely using the graphical construction. Fix a boundary nucleation rate $c\ge0$. Let the bulk nucleations
be a Poisson point process on $(0,\infty)\times(0,\infty)$ with
intensity $2\,dt\,dx$, where $t$ denotes time and $x$ position.
Independently, let the wall nucleations be a Poisson process of
rate $c$ on $(0,\infty)\times\{0\}$.

Between nucleations the height function evolves by the deterministic shift:
\begin{equation}\label{eq:deterministic-flow}
 (\phi_s h)(x)
 =\sup_{\substack{y\ge0\\|y-x|\le s}}h(y),
 \qquad s\ge0,\quad x\ge0.
\end{equation}
This flow preserves $\cX$ and satisfies
$\phi_0h=h$ and $\phi_{s+t}=\phi_s\phi_t$.
At a nucleation at position $z\ge0$, apply the nucleation map
\begin{equation}\label{eq:creation-map}
 (C_zh)(x)=h(x)+\ind_{\{x=z\}},\qquad -\infty+1=-\infty.
\end{equation}
An interior nucleation therefore creates an isolated peak of height one
above the existing profile; its two sides subsequently spread at
speed one. A wall nucleation raises the height at zero by one,
and the resulting down step moves into the half-line.

The PNG has the finite speed of propagation: notice that the evolution on $[0,b]$ up to time $T$
depends only on the initial profile on $[0,b+T]$ and the Poisson
marks in
\[
 [0,T]\times[0,b+T].
\]
There are almost surely finitely many such marks. They define a
time-homogeneous Markov process on $\cX$, denoted by
\[
 h(t;h_0),\qquad h(t,x;h_0)=h(t;h_0)(x).
\]
We write $\mathbb P_{h_0}^c$ and $\mathbb E_{h_0}^c$ for its law and
expectation when the initial profile is $h_0$.


We use the following step coordinates to describe the height function. For $h\in\cX_{\mathrm{fin}}$, set $m=h(0)$. let $0\leq y_1\le y_2\le \cdots$ be the ordered set of left/right discontinuities of \(h\), with the convention that a point $y$ is recorded twice if it is neither left nor right continuous. We define 
\[
\begin{split}
    d_i =h(y_i)-h(y_i^-), \quad \sigma_i = +1,\quad \text{ if $y_i$ is left discontinuous},\\
    d_i =h(y_i^+)-h(y_i), \quad \sigma_i = -1,\quad \text{ if $y_i$ is right discontinuous}.
\end{split}
\]
Then
\[
 h(x)
 =m+\sum_{d_i>0}d_i\ind_{\{y_i\le x\}}
    +\sum_{d_i<0}d_i\ind_{\{y_i<x\}}.
\]
So, the same profile can be described by its height at $0$ and a strict ordered locally finite list
\[
(m=h(0),\boldsymbol{y},\boldsymbol{d})=  (m;y_1,d_1;y_2,d_2;\ldots)\qquad
  y_i\leq y_{i+1},\qquad
 d_i\in\mathbb Z\setminus\{0\}.
\]
\(y_i = y_{i+1}\) is allowed only when \(d_i>0,d_{i+1}<0\).

\paragraph{\textbf{Markov generator.}}
A function $F(m,\boldsymbol{y},\boldsymbol{d})$ is \emph{local} if it depends only on the
restriction of the profile to $[0,b]$ for some finite $b$. Let $\mathcal D_{\mathrm{loc}}$ consists of
bounded continuous local functions that are continuously
differentiable on each open step-coordinate region and for which
the right derivative
\begin{equation}\label{eq:flow-derivative}
 \cL_{\mathrm{det}}F(h)
 :=\lim_{s\downarrow0}\frac{F(\phi_sh)-F(h)}{s}
\end{equation}
exists for every $h\in\cX$ and is bounded as a function of $h$. For \(F\in \mathcal D_{\mathrm{loc}}\), the generator on \(\cX\) is
\begin{equation}\label{eq:generator}
 \begin{split}
 \cL_cF(h)
 =-\sum_i\sigma_i\partial_{y_i}F(h)+2\int_0^\infty\bigl(F(C_zh)-F(h)\bigr)\,dz
      +c\bigl(F(C_0h)-F(h)\bigr).
 \end{split}
\end{equation}
The three terms describe deterministic step shift, bulk
nucleation, and wall nucleation, respectively.
The derivative $\partial_{y_i}$ moves the entire jump of size
$d_i$. The nucleation at \(0\) will increase \(m\) by \(1\).

\paragraph{\textbf{Continuity and annihilation conditions.}}
Continuity in the hypograph topology implies the following
matching identities at step annihilation and wall absorption.
When a down step of size $a$ meets an up step of size $b$, with
$a,b\in\mathbb N$ and $z>0$, continuity requires
\[
 \begin{split}
 \lim_{\varepsilon\downarrow0}
 F\bigl(m;\ldots,(z-\varepsilon,-a),(z+\varepsilon,b),\ldots\bigr)
 =F\bigl(m;\ldots,(z,b-a),\ldots\bigr).
 \end{split}
\]
When $b=a$, the coordinate on the right is omitted. If the first
coordinate is an up step of size $a$, its absorption at the wall gives
\[
 \lim_{\varepsilon\downarrow0}
 F\bigl(m;(\varepsilon,a),(y_2,d_2),\ldots\bigr)
 =F\bigl(m+a;(y_2,d_2),\ldots\bigr),\qquad a\in\mathbb N.
\]
These identities are necessary consequences of continuity.

\paragraph{\textbf{Multipoint height events.}}
Fix $n\ge1$, observation positions \( 0\le x_1<\cdots<x_n,\) and height values $\boldsymbol r=(r_1,\ldots,r_n)\in\mathbb Z^n$.
Set $\boldsymbol x=(x_1,\ldots,x_n)$ and define
\begin{equation}\label{eq:observation-event}
 f(h):=\prod_{j=1}^n\ind_{\{h(x_j)\le r_j\}}.
\end{equation}
The corresponding transition probability is
\begin{equation}\label{eqn:transition-prob}
       \mathbb E_{h_0}^c[f(h(t;h_0))]
 :=\mathbb P_{h_0}^c\bigl(h(t,x_j;h_0)\le r_j,\ 1\le j\le n\bigr).
\end{equation}

The height reached by the deterministic shift without nucleation at \(t,x\) is
\begin{equation}\label{eq:deterministic-envelope}
 r_0(t,x;h_0)
 :=(\phi_th_0)(x)
 =\sup_{\max\{0,x-t\}\le y\le x+t}h_0(y).
\end{equation}
Since nucleations only increase the profile,
$h(t,x;h_0)\ge r_0(t,x;h_0)$. This set an lower bound for values \(r_i\). We define \emph{the allowed set}
\[
 \cA_{\boldsymbol r}
 =\left\{(t,h)\in[0,\infty)\times\cX:
       r_0(t,x_j;h_0)\le r_j\text{ for every }j\right\}.
\]
If \((t,h_0)\notin \cA_{\boldsymbol r}, \mathbb E_{h_0}^c[f(h(t;h_0))] = 0\). Notice that \(f\) is not in the domain of the generator, since the function is not continuous in the upper-semicontinuous topology. This raised the technical difficulty that \eqref{eqn:transition-prob} does not satisfy the Kolmogorov equation in the classical sense; one need to use an integral version which we will introduce later.

\subsection{The multipoint Pfaffian formula}\label{subsec:kernel}
We now give a Fredholm--Pfaffian formula for the probabilities in \eqref{eqn:transition-prob}. We first define some operators that we will use later. All the operators act on functions on \(\mathbb{Z}\). Let
\begin{equation}\label{eq:free}
 \begin{gathered}
 \tau_+f(u):=f(u+1),\qquad \tau_-f(u):=f(u-1),\\
 \nabla:=\tfrac12(\tau_+-\tau_-),\qquad
 \Delta:=\tau_++\tau_--2I,\qquad S_{t,x}:=e^{2t\nabla+x\Delta}.
 \end{gathered}
\end{equation} We define the projection operators as follows:
\[
 \begin{gathered}
  1_b^a(u,v):=\ind_{\{b<u<a\}}\ind_{\{u=v\}},\qquad
  1^a(u,v):=\ind_{\{u<a\}}\ind_{\{u=v\}},\qquad
  1_b(u,v):=\ind_{\{u>b\}}\ind_{\{u=v\}},\\
  \overline{1}^{a}(u,v):=\ind_{\{u\le a\}}\ind_{\{u=v\}},\qquad
  \underline{1}_{b}(u,v):=\ind_{\{u\ge b\}}\ind_{\{u=v\}}.
 \end{gathered}
\]
Thus, a bar means that the endpoint is included.

We introduce the following weighted \(\ell^2\) space.
For $\eta\in\mathbb R$, define
\begin{equation}\label{eqn:weightedspace}
      \ell^2_\eta(I)
 :=
 \left\{f: I\to\mathbb C:
   \|f\|_\eta^2
   :=\sum_{u\in I}e^{2\eta u}|f(u)|^2<\infty
 \right\},
 \qquad M_\eta f(u):=e^{\eta u}f(u).
\end{equation}
When \(I=\mathbb{Z}\), it will be omitted. We have 
\[
 \|\tau_+\|_{\ell^2_\eta}=e^{-\eta},
 \qquad
 \|\tau_-\|_{\ell^2_\eta}=e^\eta.
\]

For $|c|<1$, define the boundary operator:
\[
 H=H_c:=(I-c\tau_+)^{-1}\nabla(I-c\tau_-)^{-1},
\]
whose integral kernel is 
\begin{equation}\label{eq:H-kernel}
 H_c(u,v)=
 \begin{cases}
  \tfrac12c^{v-u-1},&v>u,\\
  0,&v=u,\\
  -\tfrac12c^{u-v-1},&v<u.
 \end{cases}
\end{equation}
Since each entry is a polynomial
in $c$, \eqref{eq:H-kernel} defines its continuation to every
 $c\geq 0$. We define two distinct one-sided
inverse kernels:
\begin{equation}\label{eq:A}
 \begin{split}
 H_-^{-1}&:=2\sum_{k\ge0}\tau_-^{2k+1}
                  (I-c\tau_+)(I-c\tau_-),\\
H_+^{-1}&:=-2\sum_{k\ge0}\tau_+^{2k+1}
                  (I-c\tau_+)(I-c\tau_-).
 \end{split}
\end{equation}
Notice that 
$H_-^{-1}(u,v)$ is supported on \(u\geq v\) and $H_+^{-1}(u,v)$ is supported on \(u\leq v\). we have the following adjoint relation:
\begin{equation}\label{eq:H-transposes}
 H^*=-H,\qquad (H_-^{-1})^*=-H_+^{-1}.
\end{equation}

For $\eta>0$ with $ c e^\eta<1$, the operator $H_c$
is bounded on both $\ell^2_{-\eta}$ and $\ell^2_\eta$.
The series defining $H_-^{-1}$ converges in operator norm
on $\ell^2_{-\eta}$, whereas the series defining
$H_+^{-1}$ converges in operator norm on $\ell^2_\eta$.
On these respective spaces,
\begin{equation}\label{eq:inverse-domains}
       \begin{aligned}
 H_-^{-1}H_c=H_cH_-^{-1}=I
 &\qquad\text{on }\ell^2_{-\eta},\\
 H_+^{-1}H_c=H_cH_+^{-1}=I
 &\qquad\text{on }\ell^2_\eta.
 \end{aligned}
\end{equation}
For \(c\geq 1\), the composition is divergent, and one should interpret the formula as an analytic extension of \(0\leq c<1\).

Let $X$ be the continuous-time random walk that has \(+1\) or \(-1\) jumps at rate one, using its lower-semicontinuous
version at jump times. The generator of the process is $\Delta$. For $0\le a\le b$, define
\[
 \begin{split}
 P_{a,b}^h(u,v)
 &:=\mathbb P_{a,u}\bigl(X(y)\le h(y)\text{ for some }y\in[a,b],
                         \ X(b)=v\bigr),\\
 Q_{a,b}^h&:=e^{(b-a)\Delta}-P_{a,b}^h.
 \end{split}
\]
Here $\mathbb P_{a,u}$ is the law of the random walk \(X\) started at height $u$ and
spatial coordinate $a$. The kernel $Q$ gives the transition probabilities restricted to path staying
strictly above $h$ on the closed interval \([a,b]\). Consequently,
\[
 Q_{a,d}^h=Q_{a,b}^hQ_{b,d}^h\quad(a\le b\le d),\qquad
 P_{a,a}^h=\overline{1}^{h(a)},\qquad Q_{a,a}^h=1_{h(a)}.
\]
Let
\[
 T_{a,b}:=e^{a\Delta}P_{a,b}^he^{-b\Delta}H_-^{-1},
 \qquad T_b:=T_{0,b}.
\]
For an empty interval $a>b\ge0$, we set $P_{a,b}^h=T_{a,b}=0$.
The dependence of $T$ on $h,c$ is suppressed.

For nonzero real $p,q$, define 
\begin{equation}\label{eq:Wpq}
 W_{p,q}=
 \begin{cases}
  T_q-T_p^*+T_p^*HT_q,&p,q>0,\\
  T_{-p,q},&p<0<q,\\
  -T_{-q,p}^*,&q<0<p,\\
  0,&p,q<0.
 \end{cases}
\end{equation}
Notice that \(W_{p,q}\) has the relation:
$W_{p,q}^*=-W_{q,p}$.

\paragraph{\textbf{The multipoint kernel.}}
Now we define kernel \(K\) on space \(\HH=\ell^2(\{1,\ldots,n\}\times\mathbb Z;\mathbb C^2)\). Fix $0\le x_1<\cdots<x_n$ and $r_1,\ldots,r_n\in\mathbb Z$. For $t>0, i,j\in \{1,\ldots,n\}$, initially with $t\ne x_j$, let
\begin{equation}\label{eqn:Vdef}
       \mathcal S_j=\begin{pmatrix}
 S_{t,-x_j} & 0\\
 0 & HS_{t,x_j}
 \end{pmatrix},\qquad
 V_{ij}=\begin{pmatrix}
  W_{t-x_i,t-x_j}&W_{t-x_i,t+x_j}\\
  W_{t+x_i,t-x_j}&W_{t+x_i,t+x_j}
 \end{pmatrix}.
\end{equation}
Let \(K_{ij}=R_{ij}+\widetilde K_{ij}\), where
\begin{equation}\label{eq:skew-kernel}
 \begin{split}
 R_{ij}&=1_{r_i}\begin{pmatrix}
  \ind_{\{j<i\}}e^{(x_i-x_j)\Delta}&e^{x_i\Delta}He^{x_j\Delta}\\
  0&\ind_{\{i<j\}}e^{(x_j-x_i)\Delta}
 \end{pmatrix}1_{r_j},\\
 \widetilde K_{ij}&=1_{r_i}J_0\mathcal S_i^*V_{ij}\mathcal S_j1_{r_j},\quad J_0 = \begin{pmatrix}
      0 &1 \\
      -1 & 0
 \end{pmatrix}.
 \end{split}
\end{equation}
Notice that both $R$ and $\widetilde K$ include the projection operators. Let \(J\) be the operator with kernel \(J_{ij}(x,y) = \delta_{i=j}\delta_{x=y}J_0\). The kernel \(JK\) satisfies the relation: $(JK)^*=-JK$.

For a skew $2\times2$ matrix kernel $A$ on $\mathcal{H}$, its Fredholm-Pfaffian is defined by \cite[Section~8]{Rains}
\begin{equation}\label{eq:pf-definition}
 \Pf(J+A):=1+\sum_{m=1}^{\infty}\frac1{m!}
   \sum_{z_1,\ldots,z_m\in E}
       \Pf\bigl([A(z_a,z_b)]_{a,b=1}^m\bigr),
\end{equation}
We define $\Fhat_n(t,h):=\Pf(J+JK(t,h))$.

At $t=0$ or $t=x_j$, the value of $\Fhat_n(t,h)$ is defined by the fixed-profile
right-time limit of the complete grouped kernel, whose existence is
proved in Proposition~\ref{thm:transmission}.

\begin{theorem}\label{thm:multipoint}
For every $c\ge0$, $t\ge0$, initial profile $h\in\cX$ as defined in \eqref{eqn:state-space}, and the observation point \(0\leq x_1<\cdots<x_n\) and \(r_1,\ldots,r_n \in \mathbb{Z}\),
\begin{equation}\label{eq:multi-result}
 \mathbb P_h^c\bigl(h(t,x_j)\le r_j,\ 1\le j\le n\bigr)
 =\Fhat_n(t,h)=\Pf(J+JK(t,h)).
\end{equation}
\end{theorem}

In order to show that the Pfaffian is well-defined, we prove the following trace-class result in the Appendix~\ref{sec:analytic}.

\begin{proposition}\label{prop:trace-class}
      Let $M_\eta f(u)=e^{\eta u}f(u)$ and \( 0\leq c<C_0\) for some constant \(C_0\), choose
\begin{equation}\label{eq:weights}
 \omega=\log\max(1,C_0)<\eta_1<\cdots<\eta_n,\qquad
 \mathcal M_i=\diag(M_{\eta_i},M_{-\eta_i}),\qquad
 \mathcal M=\diag(\mathcal M_1,\ldots,\mathcal M_n).
\end{equation}
Define the conjugated kernel by
\begin{equation}\label{eq:weighted-kernels}
 A:=\mathcal M(JK)\mathcal M.
\end{equation}
For $h\in\cX$, \(A\) is a trace-class operator, so the Fredholm Pfaffian is well-defined.
\end{proposition}

\paragraph{\textbf{A different presentation of the kernel.}} We present another way of writing the kernel. Define
\begin{equation}\label{eq:Lpm}
 L_+=e^{t(-2\nabla-\Delta)}=e^{2t(I-\tau_+)},\qquad
 L_-=e^{t(2\nabla-\Delta)}=e^{2t(I-\tau_-)},\qquad
 L_+^*=L_-.
\end{equation}The important observation is that the sign in the \(S_{t,\pm x}\) factor, matches the sign of \(x\) in \(e^{-b\Delta}\) factor comes from \(T_{a,b}\). Thus, we have the following cancellation:
\begin{equation}\label{eq:triangular-right}
 \begin{split}
 e^{-(t-x_j)\Delta}H_-^{-1}S_{t,-x_j}
     &=e^{t(2\nabla-\Delta)}H_-^{-1}=L_-H_-^{-1},\\
 e^{-(t+x_j)\Delta}H_-^{-1}HS_{t,x_j}
     &=e^{t(2\nabla-\Delta)}=L_-.
 \end{split}
\end{equation}
For \(S_{-t,\pm x}\), the same cancellation is true as the adjoint of \eqref{eq:triangular-right}.
We write out the kernel after the cancellations in different regime. Let 
\begin{equation}\label{eq:G-exact}
 G_{j,-}:=P_{0,t-x_j}^hL_-H_-^{-1}=T_{t-x_j}S_{t,-x_j},(t>x_j)\qquad
 G_{j,+}:=P_{0,t+x_j}^hL_-=T_{t+x_j}HS_{t,x_j}.
\end{equation}
When \(t-x_i>0\) for all \(i=1,\ldots,n\),
\begin{equation}\label{eq:FG-terms}
\begin{aligned}
 S_{-t,-x_i}W_{t-x_i,t-x_j}S_{t,-x_j}
 &=S_{-t,-x_i}G_{j,-}-G_{i,-}^*S_{t,-x_j}+G_{i,-}^*HG_{j,-},\\
 S_{-t,-x_i}W_{t-x_i,t+x_j}HS_{t,x_j}
 &=S_{-t,-x_i}G_{j,+}-G_{i,-}^*HS_{t,x_j}+G_{i,-}^*HG_{j,+},\\
 -HS_{-t,x_i}W_{t+x_i,t+x_j}HS_{t,x_j}
 &=-HS_{-t,x_i}G_{j,+}-G_{i,+}^*HS_{t,x_j}+G_{i,+}^*HG_{j,+}.
\end{aligned}
\end{equation}
When \(x_i-t>0,t-x_j>0\), $W_{t-x_i,t\pm x_j} = T_{x_i-t,t\pm x_j}$.
\begin{equation}\label{eq:positive-grouped}
\begin{aligned}
 S_{-t,-x_i}T_{x_i-t,t-x_j}S_{t,-x_j}&=L_+P_{x_i-t,t-x_j}^hL_-H_-^{-1},\\
 S_{-t,-x_i}T_{x_i-t,t+x_j}HS_{t,x_j}&=L_+P_{x_i-t,t+x_j}^hL_-.
\end{aligned}
\end{equation}

The advantage of the first form in \eqref{eqn:Vdef} and \eqref{eq:skew-kernel} is that it is more symmetric, and is similar to the full-space formula. The advantage of the form \eqref{eq:FG-terms} and \eqref{eq:positive-grouped} is that \(L_+\) is upper triangular and \(L_-\) is lower triangular, which is a key fact we will use in the following proof. 

\paragraph{\textbf{Convergence of kernel under $1:2:3$ scaling.}}

We explain the kernel limit under the $1{:}2{:}3$ scaling, leaving
the uniform estimates needed for convergence of the Fredholm
Pfaffian for future work. Set
\begin{equation}\label{eq:PNG-fixed-point-scaling}
 t^\varepsilon=\varepsilon^{-3/2}t,\qquad
 x^\varepsilon=\varepsilon^{-1}x,\qquad
 c^\varepsilon=1-\varepsilon^{1/2}\rho,\qquad
 \mathfrak h^\varepsilon(t,x)
 =\varepsilon^{1/2}
   \bigl[h(t^\varepsilon,x^\varepsilon;h_0^\varepsilon)
         -2t^\varepsilon\bigr].
\end{equation}
Here the PNG process has boundary rate $c^\varepsilon$, and we
assume that
$\varepsilon^{1/2}h_0^\varepsilon(\varepsilon^{-1}\,\cdot)
 \to\mathfrak h_0$
locally in the upper-semicontinuous topology, with a common linear
upper bound.

We recall the operators entering the half-space fixed-point
formula of~\cite{Zhang24}. Let $\partial_u \to \mathbf D$ and define
\begin{equation}\label{eq:PNG-limit-boundary}
 \mathbf H(u,v)
 =\frac12\operatorname{sgn}(v-u)e^{-\rho|u-v|},
 \qquad
 \mathbf H^{-1}
 =\mathbf D^{-1}(\rho-\mathbf D)(\rho+\mathbf D),
 \qquad
 \mathbf D^{-1}(u,v)=\ind_{\{u>v\}}.
\end{equation}
The free and boundary-modified Airy operators are
\begin{equation}\label{eq:PNG-limit-Airy}
 \begin{aligned}
 \mathbf S_{t,x}=e^{t\mathbf D^3/3+x\mathbf D^2},\qquad
 \mathbf S_{\mathbf H;t,x}=\mathbf H\mathbf S_{t,x}.
 \end{aligned}
\end{equation}

Let $B$ be Brownian motion with generator $\mathbf D^2$, and let
\[
 P_{a,b}^{\mathrm{hit}(\mathfrak h_0)}(u,v)
 =\mathbb P_{a,u}\bigl(B(y)\le \mathfrak h_0(y)\text{ for some }y\in[a,b],
                         \ B(b)=v\bigr).
\]

The limit of the PNG kernel in \eqref{eq:skew-kernel} is visible term by term.
Under height rescaling, the shift operators $\tau_\pm \longrightarrow e^{\pm\varepsilon^{1/2}\mathbf D}$, so
\[
 \nabla \longrightarrow \varepsilon^{1/2}\mathbf D
          +\frac{\varepsilon^{3/2}}{6}\mathbf D^3
          +O(\varepsilon^{5/2}),
 \qquad
 \Delta \longrightarrow \varepsilon\mathbf D^2+O(\varepsilon^2).
\]
Consequently, after removing the translation by
$2t^\varepsilon$,
\[
 2t^\varepsilon\nabla
       -2t^\varepsilon\varepsilon^{1/2}\mathbf D
       \pm x_j^\varepsilon\Delta
 \longrightarrow
 \frac{t}{3}\mathbf D^3\pm x_j\mathbf D^2.
\]
This gives the Airy operators. Likewise,
\[
 \varepsilon^{-1/2}(I-c^\varepsilon\tau_\pm)
       \longrightarrow \rho\mp\mathbf D,
 \qquad
 \varepsilon^{1/2}H_{c^\varepsilon}\longrightarrow\mathbf H,
 \qquad
 \varepsilon^{-1/2}H_-^{-1}\longrightarrow\mathbf H^{-1}.
\]

For the hitting terms, the rescaled walk
$\varepsilon^{1/2}X(\varepsilon^{-1}y)$ converges to $B(y)$.
Importantly, the two different PNG hitting horizons both tend
to infinity:
\begin{equation}\label{eq:PNG-hitting-horizons}
 \varepsilon(t^\varepsilon\pm x_j^\varepsilon)
 =\varepsilon^{-1/2}t\pm x_j\longrightarrow\infty.
\end{equation}
So,
\[
T_{0,x\pm x_j} \longrightarrow \lim_{L\to \infty}P_{0,L}^{\mathrm{hit} (\mathfrak h_0)}e^{-L\mathbf D^2}\mathbf{H}^{-1}.
\]
The identification with the half-space KPZ fixed point kernel in \cite{Zhang24} can be made easily.

\subsection{Proof strategy and main contributions}\label{subsec:strategy}
The formula is guessed based on \cite{MQR,SI04,Zhang24}. The first paper has the full-space PNG hitting kernel, which contains the initial condition; the second paper gives the multipoint distribution for half-space PNG starting from narrow wedge initial condition, for rate \(c=0,1\); the third paper gives the multipoint distribution formula for the half-space TASEP, which contains the half-space reflection mechanism for the hitting kernel. In our formula, the initial
profile enters through the hitting kernels $P_{a,b}^h$, the boundary
rate through $H_c$, and the wall geometry through the kernels $W_{p,q}$. We establish the formula by verifying an integral version of Kolmogorov equation. This verification identifies the proposed Pfaffian
with the transition probability.

Although the formulas share similarity, it is \emph{not} directly the formula one get for the most natural choice. In particular, the reflection cancel mechanism is not exactly the same as in \cite{Zhang24}, and the form of \(W_{p,q}\) in \eqref{eq:Wpq} operator is particular to the half-space PNG, and considered one of the main contribution of the paper.

A second contribution is that the paper develop some different tools for PNG, which solve the same technical difficulties appeared in \cite{MQR}. First, due to the fact that indicator is not in the domain of the generator, one cannot prove the formula satisfies the classical Kolmogorov equation. \cite{MQR} introduce a weak version of the Kolmogorov equation with test functions. In this paper, we used an integral version of the Kolmogorov equation. See Proposition~\ref{prop:mild-unique}. Second, for \((t,h)\notin \mathcal{A}_{\boldsymbol{r}}\), left-hand side of \eqref{eq:multi-result} is \(0\), and one need to show the Pfaffian formula is \(0\). In \cite{MQR}, they introduce the explicit cut-off function to avoid the problem. In this paper, we show that directly.

\subsection{Outline of the paper}
In Section~\ref{sec:regular}, we prove the classical Kolmogorov equation at places where the kernel is differentiable. In Section~\ref{sec:one-FV}, we establishes the
support, the initial condition, and a uniform bound of our Pfaffian formula. In Section~\ref{sec:weak}, we prove that our formula satisfies the integral version of the Kolmogorov equation, which uniquely identifies our formula as the transition probability.
Appendix~\ref{sec:analytic} provides the trace norm bounds of the kernel, which make our Fredholm Pfaffian well-defined.
\subsection{AI Claim}
The project began in early 2025. The authors conjectured the formula
and verified the classical Kolmogorov equation
(Theorem~\ref{thm:classical}). We encountered difficulties in
establishing the weak formulation and proving that the Pfaffian
vanishes outside the admissible region.
Theorem~\ref{thm:FV} was proved entirely by AI.
Theorems~\ref{thm:initial} and~\ref{thm:verification-mild}
were proved with AI assistance. Almost all of the manuscript
was written by the authors.

\subsection*{Acknowledgments}
We thank the organizers of the
\href{https://lpetrov.cc/vipss2024/}{Virginia Integrable Probability
Summer School 2024}. This project began following our participation
in the school and was inspired by Jeremy Quastel's mini-course on
the KPZ fixed point from the perspective of polynuclear growth.
We also thank Jeremy Quastel for helpful discussions and explanations
of his joint work with Matetski and Remenik~\cite{MQR}.

\section{Kolmogorov backward equation at regular points}\label{sec:regular}

The main goal of the section is to prove the Kolmogorov backward equation for the Pfaffian on the regular allowed
set, stated in Theorem~\ref{thm:classical}. We will first define the regular allowed set.

\begin{definition}\label{def:regular}
For the fixed positions \(\boldsymbol x = (x_1,\ldots,x_n)\), let $\mathcal R_{\boldsymbol x}$
be the set of pairs $(t,h)$ with $t>0$ and $h\in\cX_{\mathrm{fin}}$
satisfying the following conditions on $[0,x_n+t]$.
\begin{enumerate}
      \item \(h\) has no jump at $|t-x_j|$, $t+x_j$ or \(0\);
      \item \(t\ne x_j,\qquad 2t\ne|x_i-x_j|\ (i\ne j),\qquad
 2t\ne x_i+x_j\quad(1\le i,j\le n)\).
 \item \(h\) has no isolated peak, i.e., \(h(z)=\max\{h(z^+),h(z^-)\}\).
\end{enumerate}
\end{definition}
The first condition makes hitting probability \(P_{a,b}^h\) differentiable at endpoints; The second condition ensures \(W_{p,q}\) remains in the same regime. The third one is simply to make the argument easier, since all the isolated peaks disappears after a small time, so they can be obtained by taking right limit. The intersection
$\mathcal R_{\boldsymbol x}\cap\cA_{\boldsymbol r}$ is the
\emph{regular allowed set}.

 Since
$\partial_t\mathcal S_j=2\nabla\mathcal S_j$ and
$\partial_t\mathcal S_i^*=-2\mathcal S_i^*\nabla$, with $\nabla=\diag(\nabla,\nabla)$ and these factors are independent of $h$,
\[
 (\partial_t-\cL_c)(\mathcal S_i^*V_{ij}\mathcal S_j)
 =\mathcal S_i^*\bigl((\partial_t-\cL_c)V_{ij}+2[V_{ij},\nabla]\bigr)
    \mathcal S_j.
\]
Define 
$\scD V_{ij}=(\partial_t-\cL_c)V_{ij}+2[V_{ij},\nabla]$. Surprisingly, \(\scD V_{ij}\neq 0\). In Section~\ref{subsec:DVcomputation}, we compute \(\scD V_{ij}\) in Proposition~\ref{prop:folded-residual} and we show in
Section~\ref{subsec:projected-backward} that it vanishes after multiplication by outer factors \(\mathcal{S}_i,\mathcal{S}_j\) and the projections.

\subsection{Rank-one and rank-two nucleation increments}\label{subsec:DVcomputation}

The nucleation terms in \eqref{eq:generator} are finite differences
under $h\mapsto C_zh$. Lemma~\ref{lem:firsthit} computes this increment
for hitting kernels. We will use the following notation: $(f\otimes g)(u,v)=f(u)g(v)$, so
$A(f\otimes g)B=(Af)\otimes(B^*g)$.
We write $\delta_k(u)=\ind_{\{u=k\}}$.

\begin{lemma}\label{lem:firsthit}
For $h\in\cX_{\mathrm{fin}}$, $0\le a\le b$, and $z\in[a,b]$,
\begin{equation}\label{eq:first-hit-change}
 P_{a,b}^{C_zh}-P_{a,b}^h
  =Q_{a,z}^h(\delta_{h(z)+1}\otimes\delta_{h(z)+1})Q_{z,b}^h.
\end{equation}
At the endpoints,
\begin{equation}\label{eq:first-hit-endpoints}
 \begin{split}
 P_{a,b}^{C_ah}-P_{a,b}^h
   &=\delta_{h(a)+1}\otimes Q_{a,b}^h(h(a)+1,\cdot),\\
 P_{a,b}^{C_bh}-P_{a,b}^h
   &=Q_{a,b}^h(\cdot,h(b)+1)\otimes\delta_{h(b)+1}.
 \end{split}
\end{equation}
\end{lemma}
\begin{proof}
      Since $C_zh\ge h$, the difference
$P_{a,b}^{C_zh}(u,v)-P_{a,b}^h(u,v)$ is the probability
of hitting $C_zh$ but not $h$, and ending at $v$.
The two profiles differ only at $z$, where the height is
increased by one. So this is exactly
\[
 \mathbb{P}_{a,u}(X(y)>h(y)\text{ for all }y\in[a,b],\quad X(z)=h(z)+1,\quad X(b)=v),
\]
which is $Q_{a,z}^h(u,h(z)+1)Q_{z,b}^h(h(z)+1,v)$ by the Markov property. At an endpoint,
$Q_{z,z}^h=1_{h(z)}$ leaves $\delta_{h(z)+1}$ unchanged and gives
\eqref{eq:first-hit-endpoints}.
\end{proof}

The next proposition shows the difference of \(W_{p,q}\) after the increment is a skew rank-two
operator. For an operator \(A\) that depends on the configuration \(h\), we define 
\[\delta_zA(h):=A(C_zh)-A(h).\]

\begin{proposition}\label{prop:rank-two}
For $z\in[0,\infty)$, $p,q\neq 0$, define \(a_z:=Q_{0,z}^h\delta_{h(z)+1}\) and define
$\beta_p$ and $u_p$ by
\[
 \begin{split}
 \beta_p&=
 \begin{cases}
  \delta_{h(z)+1}Q_{z,p}^he^{-p\Delta}H_-^{-1},&p>0,\ z\le p,\\
  0,&\text{otherwise},
 \end{cases}\\
 u_p&=
 \begin{cases}
  (I+T_p^*H)Q_{0,z}^h\delta_{h(z)+1},&p>0,\\
  e^{-p\Delta}Q_{-p,z}^h\delta_{h(z)+1},&p<0,\ -p\le z,\\
  0,&p<0,\ z<-p.
 \end{cases}
 \end{split}
\]
Then,
\begin{equation}\label{eq:rank-two-W}
 \delta_zW_{p,q}=u_p\otimes\beta_q-\beta_p\otimes u_q.
\end{equation}
\end{proposition}
\begin{proof}
      We check the identities by cases. 

      \emph{When $p,q>0$.}
For $z\le p$, Lemma~\ref{lem:firsthit} gives
\[
 \delta_zT_p
 =Q_{0,z}^h(\delta_{h(z)+1}\otimes \delta_{h(z)+1})Q_{z,p}^he^{-p\Delta}H_-^{-1}
 =a_z\otimes\beta_p.
\]
For $z>p$, both sides vanish. By definition of \(W_{p,q}\) in \eqref{eq:Wpq},
\begin{equation}\label{eqn:deltaWpq}
       \delta_zW_{p,q}
  =\delta_zT_q-(\delta_zT_p)^*
       +T_p^*H\delta_zT_q+(\delta_zT_p)^*HT_q+(\delta_zT_p)^*H\delta_zT_q.\\
\end{equation}
Notice that 
\begin{equation}\label{eq:quadratic-jump-zero}
      (\delta_zT_p)^*H\delta_zT_q
 =\left(\sum_{u,v}a_z(u)H(u,v)a_z(v)\right)
                        \beta_p\otimes\beta_q=0.
\end{equation}
The scalar sum is absolutely convergent: the bound
$0\le a_z(u)\le e^{z\Delta}(u,h(z)+1)$ gives $\|a_z\|_1\le1$,
and $H$ has bounded entries. Exchanging $u,v$ negates the sum by
$H(v,u)=-H(u,v)$, so it is zero. The remaining terms in \eqref{eqn:deltaWpq} factor as
\[
 \begin{split}
 \delta_zW_{p,q}
 &=(I+T_p^*H)(a_z\otimes\beta_q)
       +(\beta_p\otimes a_z)(HT_q-I)=u_p\otimes\beta_q-\beta_p\otimes u_q,
 \end{split}
\]
because $(HT_q-I)^*a_z=-(I+T_q^*H)a_z=-u_q$.

\noindent
For $p<0<q$ and $-p\le z\le q$, 
\[
 \delta_zW_{p,q}=\delta_zT_{-p,q}
 =e^{-p\Delta}Q_{-p,z}^h(\delta_{h(z)+1}\otimes \delta_{h(z)+1})
                  Q_{z,q}^he^{-q\Delta}H_-^{-1}
 =u_p\otimes\beta_q.
\]
Since $\beta_p=0$, this is \eqref{eq:rank-two-W}. For \(z\) outside the
interval \([-p,q]\), $u_p=0$ and $\beta_p=0$. If $-p>q$, then for every $z$ either $z<-p$ or $z>q$,
so at least one of $u_p$ and $\beta_q$ vanishes.
Hence $u_p\otimes\beta_q=0$, in agreement with
$W_{p,q}=0$ for an empty interval. The endpoint identities in
Lemma~\ref{lem:firsthit} also cover $-p=z=q$.
For $q<0<p$, skew symmetry gives
$\delta_zW_{p,q}=-(\delta_zW_{q,p})^*=-\beta_p\otimes u_q$,
and $\beta_q=0$. If $p,q<0$, $W_{p,q}=\beta_p=\beta_q=0$.
For a nucleation at the origin, $a_0=Q_{0,0}^h\delta_{h(0)+1}=1_{h(0)}\delta_{h(0)+1}=\delta_{h(0)+1}$;
all negative-endpoint factors vanish, so the same formulas apply at
$z=0$.

\end{proof}

For a hitting kernel, recall that 
\begin{equation}\label{eq:hit-generator-expanded}
\begin{split}
 \cL_cP_{a,b}^h
 ={}&-\sum_{a<y_i<b}\sigma_i\partial_{y_i}P_{a,b}^h
       +2\int_a^b(P_{a,b}^{C_zh}-P_{a,b}^h)\,dz+c\ind_{\{a=0\}}(P_{0,b}^{C_0h}-P_{0,b}^h).
\end{split}
\end{equation}
For a constant function of height $m\in\mathbb Z\cup \{-\infty\}$, introduce the killed
operators
\begin{equation}\label{eq:killed-operators}
 \Delta_m=1_m\Delta1_m,\qquad \nabla_m=1_m\nabla1_m,
 \qquad Q_s^{(m)}=1_m e^{s\Delta_m}1_m\quad(s\ge0).
\end{equation}
The projections kill the walk when it reaches height $m$. Thus,
$Q_s^{(m)}$ is the the transition probabilities of the random walk over an interval of length \(s\), staying strictly above the constant function of height \(m\) throughout the interval. In particular $Q_0^{(m)}=1_m$.

Define the two \emph{endpoint errors}
\begin{equation}\label{eqn:endpointerrors}
      \mathcal E_a:=-\delta_{h(a)}\otimes Q_{a,b}^h(h(a)+1,\cdot),\quad \mathcal E^b:=-Q_{a,b}^h(\cdot,h(b)+1)\otimes\delta_{h(b)}.
\end{equation}
Now we compute how the generator acts on the hitting kernel.

\begin{proposition}\label{prop:hit-generator}
Let $0\le a<b$. Suppose \(h\in \cX_{\mathrm{fin}}\) has no jump at \(a,b\) and \(h\) has no isolated peaks, i.e. \(h(z) = \max\{h(z^-),h(z^+)\}\). For every $c\ge0$ and $a>0$,
\begin{equation}\label{eq:LP-interior}
 \cL_c P_{a,b}^h=2[P_{a,b}^h,\nabla]+\mathcal E_a+\mathcal E^b.
\end{equation}
For $a=0$,
\begin{equation}\label{eq:LP-origin}
 \cL_c P_{0,b}^h
 =2[P_{0,b}^h,\nabla]+(I-c\tau_-)\mathcal E_0+\mathcal E^b.
\end{equation}
Here $[A,B]=AB-BA$. Jump sizes need not be one.
\end{proposition}
\begin{proof}
Write $P=P_{a,b}^h$, $Q=Q_{a,b}^h$, and denote the three contributions
in \eqref{eq:hit-generator-expanded}, in their displayed order, by
$\mathsf M$ (motion), $\mathsf B$ (bulk), and $\mathsf W$ (wall).
The endpoints $a,b$ remain fixed throughout.

By our assumption on the height function \(h\) in \eqref{eqn:state-space}. There exists $a=z_0<z_1<\cdots<z_N<z_{N+1}=b$ such that \(h\) is constant $m_j$ on $(z_j,z_{j+1})$. Put
$B_j=Q_{z_{j+1}-z_j}^{(m_j)}$. The Markov property gives
\begin{equation}\label{eq:plateau-product}
 Q:=Q^{h}_{a,b}= B_0\cdots B_N,\qquad P:=P^{h}_{a,b}=e^{(b-a)\Delta}-Q^{h}_{a,b}.
\end{equation}
In the following proof, we will omit the explicit dependence on \(a,b,h\). By our assumption on \(h\), $h(z_j)=\max(m_{j-1},m_j)$, and the adjacent projections
$1_{m_{j-1}}1_{m_j}=1_{\max(m_{j-1},m_j)}$ impose the no-hit restriction.
The endpoint assumptions give $m_0=h(a)$ and $m_N=h(b)$.

We first consider the deterministic shift motion $\mathsf M$. 
At $z_j$, write $m=m_{j-1}$, $n=m_j$,
$\sigma_j=\operatorname{sgn}(n-m)$,
$Q_L=B_0\cdots B_{j-1}$, and $Q_R=B_j\cdots B_N$.
Since $\partial_sQ_s^{(m)}=\Delta_mQ_s^{(m)}=Q_s^{(m)}\Delta_m$,
\[
 \partial_{z_j}Q=Q_L(\Delta_m-\Delta_n)Q_R,
 \qquad -\sigma_j\partial_{z_j}P
       =\sigma_jQ_L(\Delta_m-\Delta_n)Q_R.
\]
The compressed difference is
\begin{equation}\label{eq:stack-generator-difference}
 1_m(\Delta_m-\Delta_n)1_n
 =\begin{cases}
 (1_m-1_n)\Delta1_n=\delta_n\otimes\delta_{n+1},&n>m,\\
 -1_m\Delta(1_n-1_m)=-\delta_{m+1}\otimes\delta_m,&n<m.
 \end{cases}
\end{equation}
Using the fact that $Q_L=Q_L1_m$ and $Q_R=1_nQ_R$, both cases give
\begin{equation}\label{eqn:deterministicshift}
       M_j:=-\sigma_j\partial_{z_j}P
 =Q_L(\delta_{m+1}\otimes\delta_m+
                      \delta_n\otimes\delta_{n+1})Q_R.
\end{equation}
The first summand vanishes when $n>m$, and the second when $n<m$.
Thus $\mathsf M=\sum_{j=1}^NM_j$. The sum is
zero when $N=0$.

Now we consider the bulk nucleation. First consider a constant interval of height $m$ and length $s$.
Set $B_+=1_m\tau_+1_m$ and $B_-=1_m\tau_-1_m$.
The relations $B_+B_-=1_m$ and
$B_-B_+=1_m-\delta_{m+1}\otimes\delta_{m+1}$ imply
\[
 [\Delta_m,\nabla_m]=-[B_+,B_-]
                  =-\delta_{m+1}\otimes\delta_{m+1}.
\]
Consequently,
\[
 \frac{d}{dz}\bigl(Q_z^{(m)}\nabla_mQ_{s-z}^{(m)}\bigr)
 =Q_z^{(m)}[\Delta_m,\nabla_m]Q_{s-z}^{(m)}
 =-Q_z^{(m)}(\delta_{m+1}\otimes\delta_{m+1})Q_{s-z}^{(m)}.
\]
Integrating from $0$ to $s$, using $Q_0^{(m)}=1_m$, gives
\begin{equation}\label{eq:plateau-duhamel}
 2\int_0^s Q_z^{(m)}(\delta_{m+1}\otimes\delta_{m+1})Q_{s-z}^{(m)}\,dz
 =-2[Q_s^{(m)},\nabla_m].
\end{equation}
By Lemma~\ref{lem:firsthit}, this is exactly the bulk nucleation contribution on
the interval. To express it using the full operator $\nabla$, retain the
two edges between heights $m$ and $m+1$:
\[
 [Q_s^{(m)},\nabla]=[Q_s^{(m)},\nabla_m]
 -\tfrac12Q_s^{(m)}(\delta_{m+1}\otimes\delta_m)
 -\tfrac12(\delta_m\otimes\delta_{m+1})Q_s^{(m)}.
\]
Substituting into \eqref{eq:plateau-duhamel} yields
\[
 \begin{split}
 2\int_0^s Q_z^{(m)}(\delta_{m+1}\otimes\delta_{m+1})Q_{s-z}^{(m)}\,dz
 ={}&-2[Q_s^{(m)},\nabla]-Q_s^{(m)}(\delta_{m+1}\otimes\delta_m)
 -(\delta_m\otimes\delta_{m+1})Q_s^{(m)}.
 \end{split}
\]
For the full interval, apply this formula on each constant subinterval and insert
the surrounding factors from \eqref{eq:plateau-product} and take the summation. The commutator
terms telescope:
\[
 -2\sum_{j=0}^N B_0\cdots B_{j-1}[B_j,\nabla]B_{j+1}\cdots B_N
 =-2[Q,\nabla]=2[P,\nabla],
\]
where empty products are identities and
$[e^{(b-a)\Delta},\nabla]=0$. The boundary terms are
\[
 \begin{aligned}
 \text{at }a:\quad&-(\delta_{m_0}\otimes\delta_{m_0+1})Q=\mathcal E_a,\\
 \text{at }b:\quad&-Q(\delta_{m_N+1}\otimes\delta_{m_N})=\mathcal E^b,\\
 \text{at }z_j:\quad&-Q_L(\delta_{m+1}\otimes\delta_m+
                       \delta_n\otimes\delta_{n+1})Q_R=-M_j.
 \end{aligned}
\]
Hence the bulk term in \eqref{eq:hit-generator-expanded} is
\[
 \mathsf B=2[P,\nabla]+\mathcal E_a+\mathcal E^b-\sum_{j=1}^NM_j.
\]
Lastly, for a nucleation at the wall, if $a>0$, creation at zero leaves the hitting interval unchanged, so
$\mathsf W=0$. If $a=0$, put $m=h(0)$ and
$q_b(v)=Q_{0,b}^h(m+1,v)$. Lemma~\ref{lem:firsthit} gives
\[
 \mathsf W=c(P_{0,b}^{C_0h}-P_{0,b}^h)
          =c\delta_{m+1}\otimes q_b=-c\tau_-\mathcal E_0,
\]
because $\mathcal E_0=-\delta_m\otimes q_b$ and
$\tau_-\delta_m=\delta_{m+1}$.

\noindent
Summing the contributions together, the internal terms cancel:
\[
 \mathsf M+\mathsf B
 =\sum_jM_j+2[P,\nabla]+\mathcal E_a+\mathcal E^b-\sum_jM_j
 =2[P,\nabla]+\mathcal E_a+\mathcal E^b.
\]
Adding $\mathsf W=0$ proves \eqref{eq:LP-interior}; adding
$\mathsf W=-c\tau_-\mathcal E_0$ when $a=0$ replaces
$\mathcal E_0$ by $(I-c\tau_-)\mathcal E_0$ and proves
\eqref{eq:LP-origin}.
\end{proof}
Now we want to investigate how \(\partial_t\) acts on \(P^h_{a,b}\). Recall from \eqref{eq:Wpq} that \(a,b\) can be \(t\pm x_i\) or \(x_i-t\).
\begin{proposition}
\label{prop:moving-endpoints}
Assume $h$ is fixed and it is constant near the positive endpoints
of $0\le a<b$. Then
\begin{equation}\label{eq:endpoint-single}
 -\partial_aP_{a,b}-\Delta P_{a,b}=-\mathcal E_a
 \quad(a>0),\qquad
 \partial_bP_{a,b}-P_{a,b}\Delta=-\mathcal E^b.
\end{equation}
Consequently, along differentiable endpoints $0<a(t)<b(t)$,
\begin{equation}\label{eq:moving-hit-chain-rule}
 \frac{d}{dt}P_{a(t),b(t)}^h
 =-a'(t)\Delta P_{a,b}^h+b'(t)P_{a,b}^h\Delta
       +a'(t)\mathcal E_a-b'(t)\mathcal E^b.
\end{equation}
If $a=0$, only the terms containing $b'$ remain.
\end{proposition}
\begin{proof}
First assume the profile has the plateau representation
\eqref{eq:plateau-product}.
Write $m_b=h(b)$. Differentiation of the last factor of
\eqref{eq:plateau-product} gives
$\partial_bQ=Q\Delta_{m_b}$. Consequently
\[
 \partial_bP-P\Delta
 =Q(\Delta-\Delta_{m_b})
 =Q(\delta_{m_b+1}\otimes\delta_{m_b})=-\mathcal E^b.
\]
In the second equality, $Q=Q1_{m_b}$ removes every row of
$\Delta-\Delta_{m_b}$ except the downward edge from $m_b+1$ to
$m_b$. Similarly, for $a>0$ and $m_a=h(a)$,
$\partial_aQ=-\Delta_{m_a}Q$, and hence
\[
 -\partial_aP-\Delta P
 =(\Delta-\Delta_{m_a})Q
 =(\delta_{m_a}\otimes\delta_{m_a+1})Q=-\mathcal E_a.
\]
The same argument applies to any constant short interval next to each
endpoint: split it off using the constrained Markov property, leaving
the rest of the profile fixed. Thus internal peaks or stacks do not
affect these endpoint identities. The ordinary chain rule gives
\eqref{eq:moving-hit-chain-rule}. The intervals in $K$ move with
physical time, whereas
\eqref{eq:LP-interior}--\eqref{eq:LP-origin} hold the endpoints fixed.
\end{proof}

Recall that \(\scD V_{ij}=(\partial_t-\cL_c)V_{ij}+2[V_{ij},\nabla]\). We compute \(\scD W_{p,q}\) in the next proposition.
\begin{proposition}
\label{prop:folded-residual}
Let $0\le c<1$ and let $p,q\ne0$ vary with $p'(t)=q'(t)=1$. Suppose \(h\in \cX_{\mathrm{fin}}\) has no jump at \(a,b\) and \(h\) has no isolated peaks, i.e. \(h(z) = \max\{h(z^-),h(z^+)\}\). For $b>0$, define
\begin{equation}\label{eq:OD}
 \mathcal D_b:=\mathcal E^b e^{-b\Delta}H_-^{-1}
 =-\bigl(Q_{0,b}^h\delta_{h(b)+1}\bigr)\otimes
       \bigl(\delta_{h(b)}e^{-b\Delta}H_-^{-1}\bigr).
\end{equation} Then
\begin{equation}\label{eq:folded-residual}
 \scD W_{p,q}=
 \begin{cases}
 -2\bigl[(I+T_p^*H)\mathcal D_q+\mathcal D_p^*(HT_q-I)\bigr],
       &p,q>0,\\[2pt]
 -2e^{-p\Delta}(\mathcal E_{-p}+\mathcal E_q)e^{-q\Delta}H_-^{-1},
       &p<0<q\\[2pt]
 2\bigl[e^{-q\Delta}(\mathcal E_{-q}+\mathcal E_p)
                           e^{-p\Delta}H_-^{-1}\bigr]^*,
       &q<0<p\\[2pt]
 0,    &p,q<0.
 \end{cases}
\end{equation}
\end{proposition}
\begin{proof}
1. We first focus on the case \(p,q>0\). Recall that \(W_{p,q} = T_q-T_p^*+T_p^*HT_q\). For $b>0$ with $b'=1$, let
\[
 \qquad l:=(I-c\tau_-)\delta_{h(0)},\qquad
 o_b:=Q_{0,b}^he^{-b\Delta}H_-^{-1}(h(0)+1,\cdot).
\]
Thus $\mathcal E_0=-\delta_{h(0)}\otimes Q_{0,b}^h(h(0)+1,\cdot)$. The left endpoint is fixed;
differentiating $T_b=P_{0,b}e^{-b\Delta}H_-^{-1}$ and commuting
the free factors with $\nabla$ gives
\begin{align}
 \scD T_b
 &=\bigl[\partial_bP_{0,b}-P_{0,b}\Delta
       -\cL_cP_{0,b}+2[P_{0,b},\nabla]\bigr]
                         e^{-b\Delta}H_-^{-1}\notag\\
 &=\bigl[-\mathcal E^b-(I-c\tau_-)\mathcal E_0
                         -\mathcal E^b\bigr]e^{-b\Delta}H_-^{-1}
   =l\otimes o_b-2\mathcal D_b.\label{eq:T-origin-residual}
\end{align}
The first $-\mathcal E_b$ comes from endpoint differentiation
\eqref{eq:endpoint-single}; the other two terms come from subtracting
the generator identity~\eqref{eq:LP-origin}. Notice that the two \(\mathcal{E}_b\) terms have the same sign rather than cancelling each other.

The definition of $H$ gives
\[
 Hl=(I-c\tau_+)^{-1}\nabla\delta_{h(0)}
   =\tfrac12\sum_{k\ge0}c^k
                      (\delta_{h(0)-1-k}-\delta_{h(0)+1-k}),
 \qquad 1_{h(0)}Hl=-\tfrac12\delta_{h(0)+1}.
\]
Only $-\delta_{h(0)+1}/2$ is supported above $h(0)$. Using
$Q_{0,b}=1_{h(0)}Q_{0,b}$ and $P_{0,b}=e^{b\Delta}-Q_{0,b}$ gives
\[
 P_{0,b}^*Hl=e^{b\Delta}Hl-Q_{0,b}^*1_{h(0)}Hl
            =e^{b\Delta}Hl+\tfrac12 Q_{0,b}^h(h(0)+1,\cdot).
\]
The positive-weight inverse identity in \eqref{eq:inverse-domains}
gives $(H_-^{-1})^*H=-I$. Hence
\begin{equation}\label{eq:wall-dressed-column}
 \begin{split}
 (I+T_b^*H)l
 &=l+(H_-^{-1})^*e^{-b\Delta}P_{0,b}^*Hl\\
 &=l+(H_-^{-1})^*Hl+\tfrac12(H_-^{-1})^*e^{-b\Delta}\big(Q_{0,b}^h(h(0)+1,\cdot)\big)
   =\tfrac12o_b.
 \end{split}
\end{equation}
Since $H^*=-H$, the transposed factor satisfies
\begin{equation}\label{eq:wall-dressed-row}
 (HT_b-I)^*l=-(I+T_b^*H)l=-\tfrac12o_b.
\end{equation}
Time differentiation, deterministic step motion, and the commutator
are derivations. Nucleation obeys
$\delta_z(AB)=(\delta_zA)B+A(\delta_zB)+(\delta_zA)(\delta_zB)$,
so $-\cL_c$ contributes the additional term
\[
 -2\int_0^\infty(\delta_zT_p)^*H(\delta_zT_q)\,dz
       -c(\delta_0T_p)^*H(\delta_0T_q)=0.
\]
Thus
\begin{equation}\label{eq:wall-product-rule}
 \begin{split}
 \scD W_{p,q}
 &=\scD T_q-(\scD T_p)^*
       +(\scD T_p)^*HT_q+T_p^*H\scD T_q\\
 &=(I+T_p^*H)\scD T_q+(\scD T_p)^*(HT_q-I).
 \end{split}
\end{equation}
Substitute \eqref{eq:T-origin-residual} into
\eqref{eq:wall-product-rule}. The two wall contributions are
\begin{equation}\label{eqn:half-spacemechanism}
       \begin{split}
 &(I+T_p^*H)(l\otimes o_q)+(o_p\otimes l)(HT_q-I)\\
 &\quad=\bigl((I+T_p^*H)l\bigr)\otimes o_q
            +o_p\otimes\bigl((HT_q-I)^*l\bigr)\\
 &\quad=\tfrac12o_p\otimes o_q-\tfrac12o_p\otimes o_q=0,
 \end{split}
\end{equation}
by \eqref{eq:wall-dressed-column}--\eqref{eq:wall-dressed-row}.
The two remaining terms are
$-2(I+T_p^*H)\mathcal D_q$ and $-2\mathcal D_p^*(HT_q-I)$,
which proves the first case of \eqref{eq:folded-residual}.

2. For $p<0<q$, assume $-p<q$ (otherwise both \(W_{p,q} =0\) and the terms in \eqref{eq:folded-residual} are also 0), we have $W_{p,q}=T_{-p,q}$. Differentiating, with $P=P_{-p,q}$, gives
\[
 \scD T_{-p,q}=e^{-p\Delta}\bigl[
 \partial_{p}P+\partial_q P-\Delta P-P\Delta
          -\cL_cP+2[P,\nabla]\bigr]e^{-q\Delta}H_-^{-1}.
\]
The two parts of the bracket are, respectively,
\[
 \begin{aligned}
 (\partial_p P-\Delta P)+(\partial_q P-P\Delta)
     &=-\mathcal E_{-p}-\mathcal E_q,\\
 -\cL_cP+2[P,\nabla]&=-\mathcal E_{-p}-\mathcal E_q,
 \end{aligned}
\]
by \eqref{eq:endpoint-single} and \eqref{eq:LP-interior}; their sum
proves the second case of \eqref{eq:folded-residual}. For $q<0<p$, use $W_{p,q}=-W_{q,p}^*$.
\end{proof}

\subsection{The cancellation in the projected kernel}\label{subsec:projected-backward}
We now prove the backward equation for $K$ by showing the terms in Proposition~\ref{prop:folded-residual} are zero, after multiplication
by the \(S\) and the projections. Recall that 
\begin{equation}
 L_+=e^{2t(I-\tau_+)},\qquad L_-=e^{2t(I-\tau_-)},\qquad L_+^*=L_-.
\end{equation}

\begin{proposition}[Projected regular backward equation]\label{prop:projected}
Let $0\le c<1$. At every
$(t,h)\in\mathcal R_{\boldsymbol x}\cap\cA_{\boldsymbol r}$,
\begin{equation}\label{eq:projected}
 (\partial_t-\cL_c)K_{ij}(t,h)=0\qquad(1\le i,j\le n).
\end{equation}
Here $K_{ij}=R_{ij}+\widetilde K_{ij}$ includes the threshold
projections in \eqref{eq:skew-kernel}.
\end{proposition}
\begin{proof}
The kernel $R_{ij}$ is independent of $t,h$, so
\begin{equation}\label{eq:residual-sandwich}
 (\partial_t-\cL_c)K_{ij}
 =1_{r_i}J_0\mathcal S_i^*(\scD V_{ij})\mathcal S_j1_{r_j}.
\end{equation}
Here $\scD$ acts entrywise. Each entry of $V_{ij}$ is a $W_{p,q}$
with $p\in\{t-x_i,t+x_i\}$ and $q\in\{t-x_j,t+x_j\}$, so
Proposition~\ref{prop:folded-residual} describes every residual.
The assumption that \((t,h)\in \mathcal{A}_{\boldsymbol{r}}\) gives $h(y)\le r_j$ for
$y\in[(x_j-t)_+,x_j+t]$.

The expansion $L_\pm=e^{2t}\sum_{k\ge0}(-2t)^k\tau_\pm^k/k!$ shows that
$L_+$ is upper triangular and $L_-$ is lower triangular.
Since $H_-^{-1}$ is lower triangular, so is $L_-H_-^{-1}$.
Consequently, for $m\le r$,
\begin{equation}\label{eq:triangular-annihilation}
 \overline{1}^mL_-1_r=\overline{1}^mL_-H_-^{-1}1_r=0,\qquad
 1_rL_+\overline{1}^m=0.
\end{equation}

For $p,q>0$, Proposition~\ref{prop:folded-residual} gives
\[
 \scD W_{p,q}
 =-2(I+T_p^*H)\mathcal D_q-2\mathcal D_p^*(HT_q-I).
\]
The right part of $\mathcal D_q$ (defined in  \eqref{eq:OD}) is
$(e^{-q\Delta}H_-^{-1})(h(q),\cdot)$. Multiplying the right \(S\) factor, \eqref{eq:triangular-right} turns it into
$L_-H_-^{-1}$ if $q=t-x_j$, or $L_-$ if $q=t+x_j$. So
\eqref{eq:triangular-annihilation} yields
\[
 \begin{split}
 \mathcal D_{t-x_j}S_{t,-x_j}1_{r_j}&=0\qquad(t>x_j),\\
 \mathcal D_{t+x_j}HS_{t,x_j}1_{r_j}&=0.
 \end{split}
\]
Thus the first residual term vanishes after right projection. The term contains \(\mathcal{D}_p^*\) vanishes for the same reason, after the left projection.

If $p<0<q$, then $p=t-x_i$ and the left factor is
$S_{t,-x_i}^*=S_{-t,-x_i}$. Let $q=t+x_j\text{ or }t-x_j$ .
For $a<b$, the residual is
\[
 -2e^{(x_i-t)\Delta}\mathcal E_{x_i-t} e^{-q\Delta}H_-^{-1}
 -2e^{(x_i-t)\Delta}\mathcal E^q e^{-q\Delta}H_-^{-1}.
\]
The left part of \(\mathcal{E}_{x_i-t}\) (defined in \eqref{eqn:endpointerrors}) is $\delta_{h(x_i-t)}$. By the transposition of
\eqref{eq:triangular-right}, it gives
\[
 1_{r_i}S_{-t,-x_i}e^{(x_i-t)\Delta}\delta_{h(x_i-t)}
   =1_{r_i}L_+\delta_{h(x_i-t)}=0,
\]
since $h(x_i-t)\le r_i$.
The second term has right part \(\delta_{h(q)}\); since $h(q)\le r_j$, the same argument gives \(0\). For $q<0<p$, interchange
$i,j$ and transpose: $\scD W_{p,q}=-(\scD W_{q,p})^*$, so the same argument shows it is \(0\).
For $p,q<0$, $W_{p,q}=0$. All entries in
\eqref{eq:residual-sandwich} therefore vanish, proving
\eqref{eq:projected}.
\end{proof}

Proposition~\ref{prop:projected} proves the backward equation for the
kernel after threshold projection has removed the endpoint errors.
To pass to its Pfaffian, time differentiation and deterministic step
motion obey the chain rule. Nucleation requires one additional fact:
by Proposition~\ref{prop:rank-two}, it changes the weighted skew kernel
by rank at most two, and the Pfaffian is exactly affine along that
change by Lemma~\ref{lem:pfcalculus}.

\begin{theorem}[Regular backward equation for the Pfaffian]\label{thm:classical}
For $0\le c<1$, at every
$(t,h)\in\mathcal R_{\boldsymbol x}\cap\cA_{\boldsymbol r}$,
\[
 (\partial_t-\cL_c)\Fhat_n(t,h)=0.
\]
\end{theorem}

\begin{proof}
      Proposition~\ref{prop:projected} gives \((\partial_{t}-\mathcal{L}_c)K=0\) pointwise. Corollary~\ref{cor:everythingtrueintracenorm} justifies differentiation and integration in trace norm.
Proposition~\ref{prop:rank-two} shows that the nucleation gives a skew rank-two perturbation of the kernel. Then Lemma~\ref{lem:pfcalculus} gives
\[
 (\partial_{t}-\mathcal{L}_c)K=0 \Rightarrow (\partial_{t}-\mathcal{L}_c)\Fhat_n(t,h)=0.
\]
\end{proof}

\section{Support, initial value, and a uniform bound}\label{sec:one-FV}\label{sec:initial}
We establish three properties of the proposed Pfaffian: it vanishes
outside the deterministic allowed set \(\mathcal{A}_{\boldsymbol{r}}\); it has the prescribed initial
value; it is bounded uniformly over all initial profiles on bounded
time intervals. Together with the regular backward equation of
Section~\ref{sec:regular}, these properties will be used in
Section~\ref{sec:weak} to verify the integral Kolmogorov equation. Throughout this section, $0\le c<1$.
\subsection{Support}\label{subsec:forbidden}
The deterministic envelope is a lower bound for the PNG height.
Thus the height event is impossible whenever
\[
 y\in[(x_k-t)\vee0,x_k+t],
 \qquad \ell=h(y)>r_k
\]
for some $k$ and $y$.
\begin{theorem}\label{thm:FV}
For $0\le c<1$, $t>0$, and $h\in\cX_{\mathrm{fin}}$,
\begin{equation}\label{eq:FV}
 (t,h)\notin\cA_{\boldsymbol r}
 \quad\Longrightarrow\quad
 \Fhat_n(t,h)=0.
\end{equation}
\end{theorem}

 We prove the corresponding vanishing of the
Pfaffian by constructing a nonzero null vector of
$I+K$. The construction is non-explicit. Since the kernels are defined on \(\ell^2(\mathbb{Z})\) with cut-offs, essentially one wants to solve an infinite system of linear equations. Thus, one tries to construct the null vector through dimension counting. We will use the following elementary abstract fact:

\begin{lemma}\label{lem:FV-surplus}
Let $E_{\rm in}$ and $E_{\rm out}$ be finite-dimensional spaces,
and let $\mathcal U$ and $\mathcal V$ be Hilbert spaces.
Suppose a bounded operator
\[
 B=
 \begin{pmatrix}
 B_{00}&B_{01}\\
 B_{10}&B_{11}
 \end{pmatrix}
 :
 E_{\rm in}\oplus\mathcal U
 \longrightarrow
 E_{\rm out}\oplus\mathcal V
\]
has a bounded invertible block $B_{11}:\mathcal U\to\mathcal V$.
Then
\[
 \dim\ker B\ge \dim E_{\rm in}-\dim E_{\rm out}.
\]
More precisely, the map
\[
 v\longmapsto
 \bigl(v,-B_{11}^{-1}B_{10}v\bigr)
\]
is an isomorphism from
$\ker(B_{00}-B_{01}B_{11}^{-1}B_{10})$ onto $\ker B$.
\end{lemma}

\begin{proof}
Assume \((v,w)\in \ker B\). The second coordinate gives
\[
 w=-B_{11}^{-1}B_{10}v.
\]
Substitution into the first coordinate gives
\[
 (B_{00}-B_{01}B_{11}^{-1}B_{10})v=0.
\]
\end{proof}
Fix $0\le c<1$. Choose \(\eta\) such that
\[
 0<\eta_1<\cdots<\eta_n<\eta,
 \qquad ce^\eta<1,
\]
where $\eta_1,\ldots,\eta_n$ are the parameters used in \eqref{eq:weights}.
Throughout this subsection, set
\[
 \mathcal H_a:=\ell^2_\eta(a,\infty),
 \qquad a\in\mathbb Z.
\]
\(\ell^2_\eta(a,\infty)\) is defined in \eqref{eqn:weightedspace}.
The next Lemma gives the invertibility of the projected maps.
\begin{lemma}\label{lem:FV-halfline}
Fix $\eta\in\mathbb R$.
For $d,\alpha,\beta\in\mathbb R$ and $N\in\mathbb Z$, the operator
\[
 A_N=1_Ne^{dI+\alpha\tau_++\beta\tau_-}1_N
 \quad\text{on }\HH_N
\]
is invertible, with
\begin{equation}\label{eq:FV-halfline-inverse}
 A_N^{-1}
 =e^{-d}(1_Ne^{-\beta\tau_-}1_N)
         (1_Ne^{-\alpha\tau_+}1_N).
\end{equation}
Moreover,
\[
 \max\{\|A_N\|,\|A_N^{-1}\|\}
 \le e^{|d|+|\alpha|e^{-\eta}+|\beta|e^\eta},
\]
uniformly in $N$.
\end{lemma}

\begin{proof}
      Since \(\tau_-\) preserves the support above \(N\), we can factor the operator \(A_N\) as
      \[A_N = e^{d}(1_N e^{\alpha \tau_+}1_N)(1_Ne^{\beta \tau_-}1_N).\]
Notice that for \(\tau =\tau_-\) or \(\tau = \tau_+\),
\[1_N(\alpha \tau)^k1_N = (1_N \alpha \tau1_N)^k.\]
So, \(1_N e^{\alpha \tau_{\pm}}1_N = e^{\alpha 1_N\tau_{\pm}1_N}\) and we get the inverse in \eqref{eq:FV-halfline-inverse}. For the bound, since $\|\tau_+\|_{\ell^2_\eta}=e^{-\eta}$,
$\|\tau_-\|_{\ell^2_\eta}=e^\eta$, and $1_N$ is a contraction,
\[
 \|A_N\|,\ \|A_N^{-1}\|
 \le \exp\{|d|+|\alpha|e^{-\eta}+|\beta|e^\eta\},
\]
independently of $N$.
\end{proof}

Now we want to construct the null vector. Suppose the deterministic bound fails at  $x_k$, i.e. there exists
\begin{equation}\label{eq:FV-witness}
 y\in[(x_k-t)\vee0,x_k+t],\qquad \ell=h(y)>r_k.
\end{equation}
Define
\begin{equation}\label{eq:FV-endpoints}
 a_i=x_i-t,\qquad b_i=x_i+t,\qquad
 m=\max\{i\in\{k,\ldots,n\}:a_i\le y\}.
\end{equation}
Then $a_i\le y$ exactly when $i\le m$,
and $b_j\ge y$ whenever $j\ge k$.

Our unknowns are functions \(f_j \in \mathcal{H}_{r_j}, j =k,\ldots, n\). Given \(f=(f_k,\ldots,f_n)\), we construct our \(2n \times 1\) null vector $F$ to be of the following form:
\begin{equation}\label{eq:coupled-embedding}
 F_{i,-}=0,\qquad
 F_{i,+}=
 \begin{cases}0,&i<k,\\f_i,&i\ge k.\end{cases}
\end{equation}
Define the bounded operator \( \mathscr B:
\bigoplus_{j=k}^n\HH_{r_j}
 \longrightarrow
\HH_\ell\oplus\bigoplus_{i=k+1}^n\HH_{r_i}\),
\begin{subequations}\label{eq:coupled-other}
\begin{numcases}{(\mathscr Bf)_i=}
 \displaystyle
 1_\ell\sum_{j=k}^n Q_{y,b_j}^hL_-f_j,
 & $i=k$,
 \label{eq:coupled-other-first}\\[3mm]
 \displaystyle
 1_{r_i}\sum_{j=k}^{i-1}e^{(x_j-x_i)\Delta}f_j,
 & $k<i\le m$,
 \label{eq:coupled-other-middle}\\[3mm]
 \displaystyle
 1_{r_i}\left[
     f_i+\sum_{j=i+1}^n e^{(x_j-x_i)\Delta}f_j
     -L_+\sum_{j=k}^n P_{a_i,b_j}^hL_-f_j
 \right],
 & $m<i\le n$.
 \label{eq:coupled-other-last}
\end{numcases}
\end{subequations}
Boundedness follows from $0\le P_{a,b}^h,Q_{a,b}^h\le e^{(b-a)\Delta}$
entrywise on nonempty intervals, and from the exponential moments of
the free kernels.

The reason for this particular system is given by the next lemma.
\begin{lemma}\label{lem:coupled-lifting}
If $\mathscr Bf=0$, then
\begin{equation}\label{eq:full-null-from-auxiliary}
 (I+K)F=0.
\end{equation}
\end{lemma}
\begin{proof}
For the convenience of discussion, we define
\[
 Z:=\sum_{j=k}^n Q_{y,b_j}^hL_-f_j,
 \qquad \Phi:=\sum_{j=k}^n S_{t,x_j}f_j.
\]
Every no-hit path at its starting position $y$ is strictly above
$h(y)=\ell$, so $Q_{y,b_j}^h=1_\ell Q_{y,b_j}^h$. Equation
\eqref{eq:coupled-other-first} therefore gives $Z=1_\ell Z=0$. For every
$0\le a\le y$, the constrained Markov property now gives
\begin{equation}\label{eq:cut-nohit-zero}
 \sum_{j=k}^nQ_{a,b_j}^hL_-f_j=Q_{a,y}^h Z=0.
\end{equation}
The identity also holds at $a=y$ and $y=0$, using
$Q_{y,y}^h=1_{h(y)}$.

Using the case $a=0$ in
\eqref{eq:cut-nohit-zero},
\begin{equation}\label{eq:sum-hit-free}
 \begin{split}
 \sum_{j=k}^nT_{b_j}HS_{t,x_j}f_j
 &=\sum_{j=k}^nP_{0,b_j}^hL_-f_j=\sum_{j=k}^ne^{b_j\Delta}L_-f_j
       -\sum_{j=k}^nQ_{0,b_j}^hL_-f_j=\Phi.
 \end{split}
\end{equation}
Here $e^{b_j\Delta}L_-=S_{t,x_j}$, since $b_j=t+x_j$.
Also $\sum_jHS_{t,x_j}f_j=H\Phi$. For any $p>0$, expanding
the definition of $W$ therefore yields
\begin{equation}\label{eq:sum-mixed-collapse}
 \begin{split}
 \sum_{j=k}^nW_{p,b_j}HS_{t,x_j}f_j
 &=\sum_jT_{b_j}HS_{t,x_j}f_j
       -T_p^*\sum_jHS_{t,x_j}f_j
       +T_p^*H\sum_jT_{b_j}HS_{t,x_j}f_j\\
 &=\Phi-T_p^*H\Phi+T_p^*H\Phi=\Phi.
 \end{split}
\end{equation}
The two terms involving $T_p^*H$ cancel after the common no-hit
equation has been used. No individual $T_p$ has been set to zero.

Only the plus components of $F$ are nonzero, so the entries used in \eqref{eqn:Vdef} are $W_{t-x_i,t+x_j}$ and $W_{t+x_i,t+x_j}$. Now we discuss all the cases.

\smallskip\noindent\emph{The minus components.}
The upper row of $J_0\mathcal S_i^*$ is
$(0,-HS_{-t,x_i})$. Its endpoint parameter is $b_i>0$. Equations
\eqref{eq:skew-kernel} and \eqref{eq:sum-mixed-collapse} give
\[
 \begin{split}
 [(I+K)F]_{i,-}
 &=1_{r_i}\sum_{j=k}^ne^{x_i\Delta}He^{x_j\Delta}f_j
                    -1_{r_i}HS_{-t,x_i}\Phi\\
 &=1_{r_i}\sum_{j=k}^n
       \bigl(e^{x_i\Delta}He^{x_j\Delta}
                    -HS_{-t,x_i}S_{t,x_j}\bigr)f_j=0.
 \end{split}
\]
The last equality uses commutation of the free convolution factors
and $S_{-t,x_i}S_{t,x_j}=e^{(x_i+x_j)\Delta}$. 

\smallskip\noindent\emph{The plus components for $i\le m$.}
If $t-x_i>0$, the lower row of $J_0\mathcal S_i^*$ and
\eqref{eq:sum-mixed-collapse}, with $p=t-x_i>0$, give the contribution
\[
 [\tilde{K}F]_{i,+}=-1_{r_i}S_{-t,-x_i}\Phi
       =-1_{r_i}\sum_{j=k}^ne^{(x_j-x_i)\Delta}f_j.
\]
If $t-x_i<0$, so $a_i=x_i-t\le y$, every interval $[a_i,b_j]$ with $j\ge k$ is
nonempty. Equation \eqref{eq:cut-nohit-zero} and
$P_{a_i,b_j}^h=e^{(b_j-a_i)\Delta}-Q_{a_i,b_j}^h$ instead give
\[
 \begin{split}
 [\tilde{K}F]_{i,+}=-1_{r_i}L_+\sum_{j=k}^nP_{a_i,b_j}^hL_-f_j
 &=-1_{r_i}L_+\sum_{j=k}^ne^{(b_j-a_i)\Delta}L_-f_j=-1_{r_i}\sum_{j=k}^ne^{(x_j-x_i)\Delta}f_j,
 \end{split}
\]
because $L_+e^{(b_j-a_i)\Delta}L_-=e^{(x_j-x_i)\Delta}$.
The case $a_i=0$ has been excluded by $t\ne x_i$.

In either case, the terms with $j>i$ cancel the free block
$R^{++}$, and, when $i\ge k$, the term with $j=i$ cancels the
identity applied to $f_i$. Thus
\begin{equation}\label{eq:early-residual}
 [(I+K)F]_{i,+}
       =-1_{r_i}\sum_{k\le j<i}e^{(x_j-x_i)\Delta}f_j,
       \qquad i\le m.
\end{equation}
For $i\le k$ this sum is empty. For $k<i\le m$ it is
$-(\mathscr Bf)_i$, hence zero.

\smallskip\noindent\emph{The plus components for $i>m$.}
Since $a_i>y\ge0$, we get:
\[
 [(I+K)F]_{i,+}
 =1_{r_i}\left[f_i+\sum_{j=i+1}^ne^{(x_j-x_i)\Delta}f_j
            -L_+\sum_{j=k}^nP_{a_i,b_j}^hL_-f_j\right]
 =(\mathscr Bf)_i=0.
\]
Empty intervals in this sum contribute zero by the stipulated
convention. This checks every component of $(I+K)F$.

\end{proof}
Now all we need to show is that the kernel of \(\mathscr{B}\) is non-empty.
\begin{proposition}\label{prop:auxiliary-null}
For $\mathscr B$ defined by
\eqref{eq:coupled-other-first}--\eqref{eq:coupled-other-last},
\begin{equation}\label{eq:auxiliary-null-dimension}
 \dim\ker\mathscr B\ge\ell-r_k>0.
\end{equation}
\end{proposition}
\begin{proof}
Let \(h_* =\max(h(x):x\in[0,b_n]) \). Fix $N\ge\max\{h_*,\ell,r_k,\ldots,r_n\}$. Define the finite-dimensional spaces as:
\[
 E_{\rm in}=\bigoplus_{j=k}^n\ell^2_\eta(r_j,N],
 \qquad
 E_{\rm out}=\ell^2_\eta(l,N]
          \oplus\bigoplus_{i=k+1}^n\ell^2_\eta(r_i,N].
\]
The remaining input and output spaces are copies of
$H_N=\bigoplus_{j=k}^n\ell^2_\eta(N,\infty)$. Notice that \(\dim E_{\rm in}-\dim E_{\rm out} =l-r_k>0 \), which is the key point that we can find the null vector. Let $\mathscr B_{ij}(u,v)$ be the kernel of operator $\mathscr B$. Define
\begin{equation}\label{eq:four-height-blocks}
 \begin{array}{ll}
 (B_{00})_{ij}=\overline{1}^{N}\mathscr B_{ij}\overline{1}^{N},
   &(B_{01})_{ij}=\overline{1}^{N}\mathscr B_{ij}1_N,\\[1mm]
 (B_{10})_{ij}=1_N\mathscr B_{ij}\overline{1}^{N},
   &(B_{11})_{ij}=1_N\mathscr B_{ij}1_N.
 \end{array}
\end{equation}
Thus
\[
 \mathscr B=\begin{pmatrix}B_{00}&B_{01}\\B_{10}&B_{11}\end{pmatrix}
 :E_{\rm in}\oplus H_N\longrightarrow E_{\rm out}\oplus H_N.
\]
\smallskip\noindent
We prove that \(B_{11}\) is invertible
once \(N\) is sufficiently large. 
Using $Q_{y,b_j}^h=e^{(b_j-y)\Delta}-P_{y,b_j}^h$ in the
first row, and separating the hitting terms in the other rows,
gives the exact decomposition
\begin{equation}\label{eq:high-block-decomposition}
 B_{11}=D_N-C_N.
\end{equation}
where the \(D_N\) part consists of the free random walk part:
\[
 (D_N)_{ij}:=
 \begin{cases}
  1_Ne^{2t\nabla+(x_j-y)\Delta}1_N,
       &i=0,\\
  1_Ne^{(x_j-x_i)\Delta}1_N,
       &k<i\le m,\ j<i,\\
  1_Ne^{(x_j-x_i)\Delta}1_N,
       &m<i\le n,\ j\ge i,\\
  0,&\text{otherwise},
 \end{cases}
\]
and \(C_N\) consists of all the hitting contributions:
\[
 (C_N)_{ij}:=
 \begin{cases}
  1_NP_{y,b_j}^hL_-1_N,&i=0,\\
  0,&k<i\le m,\\
  1_NL_+P_{a_i,b_j}^hL_-1_N,&m<i\le n.
 \end{cases}
\]
As before, an empty hitting interval contributes zero. We will show \(D_N\) is invertible with inverse being uniformly bounded, and \(C_N\) is small.

First, order the columns and rows of $D_N$ as
\[
 \begin{array}{ll}
 \text{columns:}&
 k,\ldots,m-1,\quad n,\ldots,m+1,\quad m,\\
 \text{rows:}&
 k+1,\ldots,m,\quad n,\ldots,m+1,\quad 0,
 \end{array}
\]
In this order $D_N$ is lower block
triangular, with diagonal blocks
\[
 \begin{gathered}
  1_Ne^{(x_{i-1}-x_i)\Delta}1_N
       \quad(k<i\le m),\qquad I\quad(m<i\le n),\qquad
  1_Ne^{2t\nabla+(x_m-y)\Delta}1_N.
 \end{gathered}
\]
Indeed, the first group of rows successively determines
$f_k,\ldots,f_{m-1}$, the second group determines
$f_n,\ldots,f_{m+1}$ in reverse order, and the last row
determines $f_m$.

 Applying Lemma~\ref{lem:FV-halfline} for the diagonal terms and \eqref{eq:free-tail} for the off-diagonal free blocks, we have
\[
 \sup_N\|D_N^{-1}\|<\infty.
\]

Next, the hitting correction tends to zero in operator norm.
The cumulative hitting bound in Lemma~\ref{lem:tails}, together
with the triangularity of $L_\pm$, gives, for every $\lambda>0$,
\begin{equation}\label{eq:auxiliary-correction-tail}
 |(C_N)_{ij}(u,v)|
 \le C_\lambda e^{-\lambda(u+v-2h_*)},
 \qquad u,v>N,
\end{equation}
with $C_\lambda$ independent of $N$. Choose $\lambda>\eta$.
Summing the absolute entries over the finitely many blocks gives
\[
 \begin{aligned}
 \|C_N\|
 &\le
 \sum_{i,j}\sum_{u,v>N}
       e^{\eta(u-v)}|(C_N)_{ij}(u,v)|\le
 C_{\lambda,\eta}e^{-2\lambda(N-h_*)}
 \longrightarrow0.
 \end{aligned}
\]
Thus, \(B_{11}\) is invertible. Applying Lemma~\ref{lem:FV-surplus} and using the fact that \(\dim E_{\rm in}-\dim E_{\rm out} =l-r_k>0 \), we obtain \eqref{eq:auxiliary-null-dimension}.

\end{proof}
\begin{proof}[Proof of Theorem~\ref{thm:FV}]
Proposition~\ref{prop:auxiliary-null} and Lemma~\ref{lem:coupled-lifting}
give $(I+K)F=0$ for the vector in
\eqref{eq:coupled-embedding}. Thus, the Pfaffian is \(0\).

\end{proof}
\subsection{The initial value}\label{subsec:clusters}
We first prove a factorization of the formulas when \(x_i,x_{i+1}\) are separated by \(2t\). Because of the finite speed of propagation, notice that \(h(T,x_i)\) depends on the initial condition \(h_0\)  only through the piece \(((x_i-t)\vee 0,x_i+t)\) and also the random points in the cone region \((t,x)\in \{(t,x):0\leq t\leq T, (x_i-t)\vee 0\le x \le x_i+t\}\). When \(x_{i+1}-x_i> 2t\), \(h(t,x_i)\) and \(h(t,x_{i+1})\) are independent. This fact is illustrated by the following separation identity. 

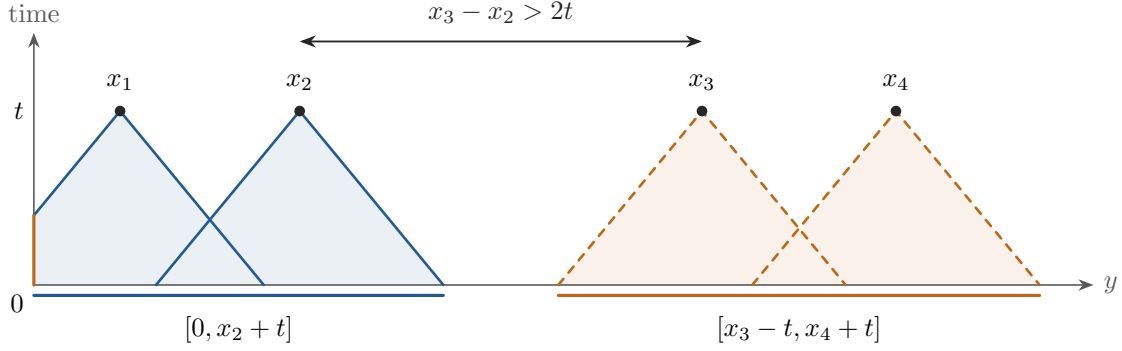
\begin{figure}[t]
\centering
\begin{tikzpicture}[
  x=.95cm,y=1.15cm,font=\small,>=Stealth,
  line cap=round,line join=round]
\tikzset{
  cone left/.style={draw=proofblue,line width=1pt},
  cone right/.style={draw=wallorange,dashed,line width=1pt},
  cone axis/.style={black!65,->,line width=.6pt}
}

\fill[proofblue!9]
  (0,0)--(3.2,0)--(1.2,2)--(0,.8)--cycle;
\fill[proofblue!9]
  (1.7,0)--(5.7,0)--(3.7,2)--cycle;
\fill[wallorange!9]
  (7.3,0)--(11.3,0)--(9.3,2)--cycle;
\fill[wallorange!9]
  (10,0)--(14,0)--(12,2)--cycle;

\draw[cone axis] (0,0)--(14.75,0) node[right] {$y$};
\draw[cone axis] (0,0)--(0,2.9) node[above] {time};
\node[left] at (0,2) {$t$};
\node[below left] at (0,0) {$0$};

\draw[cone left] (0,.8)--(1.2,2)--(3.2,0);
\draw[cone left] (1.7,0)--(3.7,2)--(5.7,0);
\draw[cone right] (7.3,0)--(9.3,2)--(11.3,0);
\draw[cone right] (10,0)--(12,2)--(14,0);

\draw[wallorange,line width=1.2pt] (0,0)--(0,.8);

\foreach \xx/\lab in {1.2/1,3.7/2,9.3/3,12/4}{
  \fill[black!85] (\xx,2) circle (2pt);
  \node[above=4pt] at (\xx,2) {$x_{\lab}$};
}

\draw[<->,black!85,line width=.7pt]
  (3.7,2.8)--(9.3,2.8)
  node[midway,above=3pt] {$x_3-x_2>2t$};

\draw[proofblue,line width=1.2pt] (0,-.12)--(5.7,-.12);
\node[below=4pt] at (2.85,-.12) {$[0,x_2+t]$};

\draw[wallorange,line width=1.2pt] (7.3,-.12)--(14,-.12);
\node[below=4pt] at (10.65,-.12) {$[x_3-t,x_4+t]$};

\end{tikzpicture}
\caption{Backward dependence cones for four observations at time $t$.}
\label{fig:separated-dependence-cones}
\end{figure}
 
For a set of observation indices $I\subseteq\{1,\ldots,n\}$, define
$P_I$ on $\ell^2(E)\times \ell^2(E)$ by
\[
 (P_If)(j,u)=\ind_{\{j\in I\}}f(j,u).
\]
Let
\begin{equation}\label{eq:positive-determinant}
 \begin{split}
 \Fhat(t,h;I)&=\Pf(J+JP_IK(t,h)P_I),\\
 D(t,h;I)&=\det\left(
 I+P_I\bigl[K^{++}_{ij}(t,h)\bigr]_{i,j=1}^n
                       P_I\right).
 \end{split}
\end{equation}
Here $K^{++}_{ij}$ is the \((+,+)\) component of the kernel. 

\begin{lemma}[Separation of observations]\label{lem:cluster}
Suppose $1\le k<n$ and $x_{k+1}-x_{k}>2t$. Put
$I_L=\{1,\ldots,k\}$ and $I_R=\{k+1,\ldots,n\}$. Then
\begin{equation}\label{eq:separated-factorization}
 \Fhat_n(t,h)=\Fhat(t,h;I_L)D(t,h;I_R).
\end{equation}
Whenever \(t<x_i\) for all \(i\),
$\Fhat(t,h;I)=D(t,h;I)$.
\end{lemma}
Figure~\ref{fig:separated-dependence-cones} illustrates how a gap
greater than $2t$ separates the backward dependence cones into
two groups.
\begin{proof}
Let $I_L=\{1,\ldots,k\}$ and $I_R=\{k+1,\ldots,n\}$, with
$x_{k+1}-x_k>2t$, and first suppose that $t\ne x_j$ for every
$j$. Choose finitely many basepoints $(i,u)$ and retain both
components $(i,u,-)$ and $(i,u,+)$ at each chosen basepoint.
This gives a paired finite principal restriction of $J+JK$.

We first identify the zero blocks that determine its possible
Pfaffian pairings. For $i\in I_R$, the gap assumption gives
\[
 x_i>x_k+2t\ge2t,
 \qquad t-x_i<0.
\]
For every observation $j$, the folded-kernel definition therefore
implies
\[
 W_{t-x_i,t-x_j}
 =
 \begin{cases}
  0,&x_j>t,\\
  T_{x_i-t,t-x_j}=0,&x_j<t.
 \end{cases}
\]
In the second case, the interval is empty because
$x_i-t>t\ge t-x_j$. So \( (J+JK)^{--}_{ij}=
 0, i\in I_R, 1\le j\le n.\)

Next, if $i\in I_R$ and $j\in I_L$, then \( i>j, x_i-t>x_j+t.\)
Thus, $W_{t-x_i,t+x_j}=T_{x_i-t,t+x_j}=0$. Since the contribution from kernel \(R\) is also \(0\), we have \( (J+JK)^{-+}_{ij}=0\). Consequently, every nonzero pairing in the Pfaffian series expansion consists
entirely of pairs within the left group and pairs within the
right group.

Reordering the entries in the Pfaffian indexed by \(i,j\in I_R\), it has the form
\[
 \begin{pmatrix}0&B\\-B^*&C\end{pmatrix},\qquad
 B=I+[K^{++}_{ij}]_{i,j\in I_R}.
\]
If \(|I_R|=q\), its Pfaffian is
$(-1)^{q(q-1)/2}\det B$. The same reordering gives that sign for
$J$, so the normalized factor is $\det B$, independently of $C$.

When \(x_i>t\) for all \(i\), \((J+JK)_{ij}^{--}= 0\) for all \(1\le i,j\le n\). Thus the same argument gives $\Fhat(t,h;I)=D(t,h;I)$. 
\end{proof}

\begin{proposition}\label{thm:initial}
For $0\le c<1$ and every $h\in\cX_{\mathrm{fin}}$,
\begin{equation}\label{eq:init-result}
 \lim_{t\downarrow0}\Fhat_n(t,h)
 =\prod_{j=1}^n\ind_{\{h(x_j)\le r_j\}}.
\end{equation}
\end{proposition}

\begin{proof}
If $h(x_k)>r_k$ for some $k$, then
\[
 r_0(t,x_k;h)\ge h(x_k)>r_k \text{ for all } t>0.
\]
Theorem~\ref{thm:FV} gives $\Fhat_n(t,h)=0$. Thus, now we suppose that $h(x_j)\le r_j$ for every
$j$.

Upper semicontinuity and the discreteness of the heights allow
us to choose $t_0>0$ such that
\[
 \begin{gathered}
 t_0<x_j\quad\text{whenever }x_j>0,\qquad
 2t_0<x_{j+1}-x_j,\\
 h(y)\le r_j
 \quad\text{for }y\in[(x_j-t_0)\vee0,x_j+t_0],
 \qquad 1\le j\le n.
 \end{gathered}
\]
Thus, for $0<t<t_0$, all \([x_i-t,x_i+t]\) are disjoint.

By Lemma~\ref{lem:cluster}, we get
\begin{equation}\label{eq:small-time-factorization}
 \Fhat_n(t,h)
 =
 \Pf(J_0+J_0 K_{11})\prod_{j=2}^n
 \det(I+K_{jj}^{++})
\end{equation}
If \(x_1>0\), the first Pfaffian has the same form as other factors; if \(x_1=0\), it needs a separate discussion.

For each determinant factor, 
\begin{equation}\label{eq:initial-singleton-det}
      \det(I+K_{jj}^{++})=\det(I-1_{r_j}L_+P_{x_j-t,x_j+t}^hL_-1_{r_j}).
\end{equation}
As \(t\downarrow 0\), using the upper semicontinuity of $h$ at $x_j$ we have
\[
 P_{x_j-t,x_j+t}^h(u,v)
 \longrightarrow
 \ind_{\{u=v\le h(x_j)\}}.
\]
Since $L_\pm(t)\to I$, we have the pointwise limit
\[
 1_{r_j}L_+P_{x_j-t,x_j+t}^hL_-1_{r_j}
 \longrightarrow
 1_{r_j}\overline1^{h(x_j)}1_{r_j}=0.
\]
To pass to trace norm, the estimate \eqref{eq:direct-tail} bounds
the absolute entries by
\[
 C_\lambda e^{-\lambda(u+v)},
 \qquad u,v>r_j,
\]
uniformly for $0<t<t_0$. This bound is summable over \((r_j,\infty)\). Dominated convergence gives
\[
 \left\|
 1_{r_j}L_+P_{x_j-t,x_j+t}^hL_-1_{r_j}
 \right\|_1\longrightarrow0.
\]
Continuity of the Fredholm determinant in
\eqref{eq:initial-singleton-det} shows that every factor with
$x_j>0$ tends to $1$.

It remains to consider the case $x_1=0$. Notice that the four entries in \(V_{11}\), defined in \eqref{eqn:Vdef}, are all \(W_{t,t}\). Thus,
\[
 P_{0,t}^h(u,v)
 \longrightarrow
 \ind_{\{u=v\le h(0)\}}.
\]
Using the operators \(L_-,H_{-}^{-1}\) and their support properties, when \(h(0)\leq r_1\)
\[
 \begin{aligned}
 P_{0,t}^hL_-1_{r_1}
 &\longrightarrow
 \overline1^{h(0)}1_{r_1}=0,\\
 P_{0,t}^hL_-H_-^{-1}1_{r_1}
 &\longrightarrow
 \overline1^{h(0)}H_-^{-1}1_{r_1}=0.
 \end{aligned}
\]

Every term of \(\tilde{K}_{11}\) contains
one of these factors or its transpose, as displayed in
\eqref{eq:FG-terms}. The common projection $1_{r_1}$ and proof of Proposition~\ref{prop:trace-class} justify the limits
through the internal sums and give
\[
 \left\|
 \mathcal M_1J_0\widetilde K_{11}(t,h)\mathcal M_1
 \right\|_1\longrightarrow0.
\]
The remaining part of the kernel is \(J_0R_{11}
 =
 \begin{pmatrix}
  0&0\\
  0&-1_{r_1}H1_{r_1}
 \end{pmatrix}.\)
By trace-norm continuity, the Fredholm Pfaffian \(\Pf(J_0+J_0 K_{11})\) goes to $1$. 

All factors in \eqref{eq:small-time-factorization} thus tend
to $1$ when $h(x_j)\le r_j$ for every $j$. Together with the
forbidden case treated at the beginning, this proves
\eqref{eq:init-result}.
\end{proof}

\subsection{A uniform bound}
We need to establish a uniform bound that is independent of the
initial profile, which will be used in the later uniqueness argument.
\begin{proposition}\label{thm:uniform}
For fixed $T,\boldsymbol x,\boldsymbol r$ and $0\le c<1$, there is
$C<\infty$ independent of \(t,h\) such that
\begin{equation}\label{eq:uniform-F}
 |\Fhat_n(t,h)|\le C\qquad
       (0\le t\le T,\ h\in\cX_{\mathrm{fin}}).
\end{equation}
\end{proposition}

\begin{proof}
First suppose $t>0$ and $t\ne x_j$ for every $j$.
Theorem~\ref{thm:FV} gives zero outside the allowed set \(\mathcal{A}_{\boldsymbol{r}}\).
For an allowed \(h\), apply Lemma~\ref{lem:cluster} whenever there exists \(x_k\) such that \(x_{k+1}-x_k>2t\). So we get 
\begin{equation}\label{eqn:Fhatfactor}
      \Fhat_n(t,h)=\Fhat(t,h;I_L)\prod_{k}D(t,h;I_k).
\end{equation}
where \(I_k\) is the index set for each group. 

Let $r_{\max}=\max_j r_j$. Proposition~\ref{prop:trace-class}, with ceiling \(h_* =r_{\max}\), bounds the
trace norm of the kernel in the first Pfaffian.
For each Fredholm determinant factors, the bounds \eqref{eq:direct-tail} and 
\eqref{eq:free-heat-sum} provide the trace norm bounds as well. Notice that only heights on these hitting
intervals enter the estimates, and these heights are bounded by \(r_{\max}\), and heights in the gaps between groups are unrestricted, thus the bound is uniform in \(h\).

Since there are at most \(n\) factors in \eqref{eqn:Fhatfactor}, we let $C_*$  be the same bound for the kernels in each factor. The inequalities
$|\Pf(J+A)|\le e^{\|A\|_1/2}$ and
$|\det(I+D)|\le e^{\|D\|_1}$ give
\[
 |\Fhat_n(t,h)|\le e^{nC_*}.
\]
If $t=x_j>0$, we have the same bound by taking a right limit of \(t\). Finally, Proposition~\ref{thm:initial} gives
$|\Fhat_n(0,h)|\le1$. This proves \eqref{eq:uniform-F}.
\end{proof}

\section{Verification and identification of the Pfaffian formula}\label{sec:weak}
In this section, we first introduce the integral formulation of the Kolmogorov equation in Proposition~\ref{prop:mild-unique} and establish its uniqueness. We then show that our formula satisfies the equation. What one needs to show is that the integrand is continuous at a non-differentiable point, see Proposition~\ref{thm:transmission}. The Bounded uniqueness and analytic continuation prove
Theorem~\ref{thm:multipoint} for all boundary rates. 

\subsection{Finite propagation and bounded uniqueness}
Fix
\begin{equation}\label{eq:weak-fixed-data}
 T>0,\qquad B>x_n+T+1.
\end{equation}
Let $\cX_B$ be the finite-valued, integer-valued upper-semicontinuous
profiles on $[0,B]$ that are constant on the open intervals of a finite
partition. At a partition point an isolated upper value is allowed.
Thus finite stacks and peaks, including at the endpoints, are included.
Define
\begin{equation}\label{eq:finite-flow}
 (\phi_s^B h)(x)=\sup_{\substack{y\in[0,B]\\|y-x|\le s}}h(y).
\end{equation}
The artificial finite-volume process alternates this flow with creation
maps $C_z$, at the marks of a Poisson random measure with intensity
$ds\,\mu_B(dz)$, where
\begin{equation}\label{eq:birth-measure}
 \mu_B(dz)=2\ind_{(0,B)}(z)\,dz+c\delta_0(dz),\qquad
 \lambda_B=\mu_B([0,B))=2B+c.
\end{equation}
Write $h^B(t;h)$ for this process and $\mathbb E_h^{B,c}$ for its
expectation. There is no source or incoming influence at $B$.
It is not the restriction of the half-space process on all of $[0,B]$;
nevertheless finite propagation gives, under the common-mark coupling,
\begin{equation}\label{eq:finite-propagation-coupling}
 h^B(s;h_0|_{[0,B]})(x)=h(s;h_0)(x)\qquad(x+s<B).
\end{equation}
The cutoff serves only to make the total nucleation rate finite.

\begin{proposition}\label{prop:mild-unique}
For every bounded Borel function $f$ on $\cX_B$ and every $c\ge0$,
the function $U(t,h)=\mathbb E_h^{B,c}[f(h^B(t;h))]$ is the unique
bounded Borel solution on $[0,T]\times\cX_B$ of
\begin{equation}\label{eq:mild}
 \begin{split}
 U(t,h)={}&e^{-\lambda_Bt}f(\phi_t^B h)+\int_0^t e^{-\lambda_Bs}
      \int_{[0,B)}U(t-s,C_z\phi_s^B h)\,\mu_B(dz)\,ds.
 \end{split}
\end{equation}
Boundedness means a single bound over both time and the entire profile
space.
\end{proposition}
\begin{proof}
Let $S_1$ and $Z_1$ be the first nucleation time and location. Their law is
\[
 \mathbb P(S_1>t)=e^{-\lambda_Bt},\qquad
 \mathbb P(S_1\in ds,Z_1\in dz)=e^{-\lambda_Bs}\,ds\,\mu_B(dz).
\]
Conditional on this time and location, the remaining marks give the same
process started from
$C_z\phi_s^B h$. This proves \eqref{eq:mild}. The expectation is
bounded by $\|f\|_\infty$ and Borel, since the graphical construction
is a finite composition of Borel maps conditional on finitely many
marks.

For uniqueness, let $D$ be the difference of two bounded solutions
with the same $f$, and put
$\|D\|_T=\sup_{0\le t\le T,\,h\in\cX_B}|D(t,h)|$.
Subtracting their equations gives
\[
 |D(t,h)|\le\|D\|_T\int_0^t\lambda_Be^{-\lambda_Bs}\,ds
            =(1-e^{-\lambda_Bt})\|D\|_T.
\]
Hence $\|D\|_T\le(1-e^{-\lambda_BT})\|D\|_T$, so $D=0$.
\end{proof}

For the rest of the direct verification, $0\le c<1$. Set
\begin{equation}\label{eq:weak-candidate}
 G(t,h)=\Fhat_n(t,h),\qquad
 f(h)=f_{\boldsymbol x,\boldsymbol r}(h)
           =\prod_{j=1}^n\ind_{\{h(x_j)\le r_j\}}.
\end{equation}
For $h\in\cX_B$, evaluate the kernel using any locally finite extension
beyond $B$. All relevant intervals for $0\le t\le T+1$ lie strictly
inside $[0,B]$, so the extension has no effect. Proposition~\ref{thm:initial}
gives $G(0,h)=f(h)$, and Proposition~\ref{thm:uniform} gives
\begin{equation}\label{eq:global-candidate-bound}
 |G(t,h)|\le C\qquad(0\le t\le T+1,\ h\in\cX_B).
\end{equation}
In order to check the integral equation, we need to check the continuity of the Fredholm Pfaffian across the exceptional times. Keep the initial profile $h$
and the total time $t$ fixed, and write
\[
 \tau_s=t-s,\qquad g_s=\phi_s h,\qquad 0\le s<t.
\]
Assume \(s\) is an exceptional time, i.e. \((t-s,\phi_s h)\notin\mathcal R_{\boldsymbol x}\). Let \(z_1,\ldots,z_N\) be the finite number of discontinuity of \(h\) in $[0,B]$. We classify all the possible scenarios for exceptional times:
\begin{enumerate}
 \item
 \emph{A change in the profile coordinates:}
 steps collide and annihilate, or an up-step reaches the wall.
 At $s=0$, an initial isolated peak or excess wall height
 begins to spread.

 \item
 \emph{A step meets a hitting endpoint:}
 \[
  z_i-\sigma_i s=(t-s)+x_j
  \qquad\text{or}\qquad
  z_i-\sigma_i s=|(t-s)-x_j|.
 \]

 \item
 \emph{The kernel \(W_{p,q}\) changes:}
 \[
  (t-s)=x_j,\qquad
  2(t-s)=|x_i-x_j|\ (i\ne j),\qquad
  2(t-s)=x_i+x_j.
 \]
 These are the zero signed endpoints and the times at which
 a mixed hitting interval opens or closes.
\end{enumerate}

\subsection{Matching and right-time limits}\label{subsec:matching}
To integrate the regular backward equation, we need its solutions to
match at exceptional configurations.
\begin{proposition}
\label{thm:transmission}
Let $0\le c<1$, and let $h\in\cX_{\mathrm{fin}}$.

\smallskip\noindent
Fix $t>0$ with $(t,h)\in\cA_{\boldsymbol r}$, and consider
the one parameter family
\[
 s\longmapsto (t-s,\phi_s h),\qquad 0\le s<t.
\]
Let $s_0\in[0,t)$ be an isolated exceptional time of this
curve, meaning that \((t-s_0,\phi_{s_0} h)\notin\mathcal R_{\boldsymbol x}\)
\[
 (t-s,\phi_s h)\in\mathcal R_{\boldsymbol x}
 \quad\text{for all }s\ne s_0\text{ sufficiently close to }s_0.
\]
Then the candidate matches continuously across $s_0$:
\begin{equation}\label{eqn:shiftcontinuity}
       \lim_{s\to s_0}\Fhat_n(t-s,\phi_s h)
 =
 \Fhat_n(t-s_0,\phi_{s_0}h).
\end{equation}
At $s_0=0$, the limit is taken from the right. Furthermore, for \(h\in \cX_{\mathrm{fin}}, t>0\),
\begin{equation}\label{eqn:rightcontinuity}
       \lim_{\varepsilon\downarrow0}
 \Fhat_n(t+\varepsilon,h)=\Fhat_n(t,h),\qquad  \lim_{\varepsilon\downarrow0}
 \Fhat_n(\varepsilon,h)
 =
 \prod_{j=1}^n\ind_{\{h(x_j)\le r_j\}}.
\end{equation}
There is no assumption on the pair \((t,h)\).
\end{proposition}

\begin{proof}
We first prove \eqref{eqn:shiftcontinuity}. We discuss the three scenarios for the exceptional time.

\smallskip\noindent
\emph{1. Collisions, wall absorption, and initial peaks.}
For a constant part of height $q$ whose width $\delta$ tends to zero,
\[
 Q_\delta^{(q)}\longrightarrow1_q,\qquad
 1_pQ_\delta^{(q)}1_r
 \longrightarrow1_{\max(p,q,r)}\quad  \text{ in operator norm.}
\]

Consequently, when a valley disappears, the limiting no-hit
product kernel (the one in \eqref{eq:plateau-product}) is exactly the restrictions of the limiting
profile. When $s\downarrow0$ and an island shrinks back to an
initial isolated peak, its height restriction remains at that
point. At the wall, the analysis is the same.

\smallskip\noindent
\emph{2. Contacts with hitting endpoints.}
All kernels in this step are evaluated at the profile $g=\phi_s h$. For $0\le a<b$, let $Q_{a,b^-}^g$ and
$Q_{a^+,b}^g$ denote the no-hit kernels on $[a,b)$ and
$(a,b]$. $P_{a,b^-}^g$ and
$P_{a^+,b}^g$ are defined analogously.
 
Upper semicontinuity gives $g(b)\ge g(b^-)$ and $g(a)\ge g(a^+)$.
Since
\[
 Q_{a,b}^g=Q_{a,b^-}^g1_{g(b)}
         =1_{g(a)}Q_{a^+,b}^g,
\]
inserting an omitted endpoint restriction gives
\begin{equation}\label{eq:endpoint-captures}
 \begin{split}
 P_{a,b}^g-P_{a,b^-}^g
 &=\sum_{g(b^-)<v\le g(b)}
       Q_{a,b^-}^g(\cdot,v)\otimes\delta_v,\\
 P_{a,b}^g-P_{a^+,b}^g
 &=\sum_{g(a^+)<v\le g(a)}
       \delta_v\otimes Q_{a^+,b}^g(v,\cdot).
 \end{split}
\end{equation}
These are precisely the possible discrepancies when a
step meets a hitting endpoint: one can hit exactly at the endpoint. 

For a right endpoint belonging to column $j$,
the assumption that \(g\in \mathcal{A}_{\boldsymbol{r}}\) gives $g(b)\le r_j$. It is followed by the kernel
$L_-$ or $L_-H_-^{-1}$, so the column vanishes because:
\[
 \overline1^{g(b)}L_-1_{r_j}=0,\qquad
 \overline1^{g(b)}L_-H_-^{-1}1_{r_j}=0.
\]
The discussion for the left endpoint is the same.

The cancellation also holds for \(W_{p,q}\). Indeed, the exact difference of
$W_{p,q}=T_q-T_p^*+T_p^*HT_q$ is
\[
 \delta W_{p,q}
 =(I+T_p^*H)\delta T_q
 +(\delta T_p)^*(HT_q-I)
 +(\delta T_p)^*H\delta T_q.
\]
After putting the \(\mathcal{S}_j\) matrix in \eqref{eqn:Vdef} around it, the first term contains the
annihilated right factor, the second contains the annihilated left
factor, and the third contains both. 

Thus endpoint contacts will not change the limiting kernel.

\smallskip\noindent
\emph{3. Changes of kernel \(W_{p,q}\).}
First suppose a hitting interval collapses to a positive
point $a$. The closed-interval convention gives
\[
 P_{a,a}^g=\overline1^{g(a)},\qquad
 T_{a,a}
 =e^{a\Delta}\overline1^{g(a)}
      e^{-a\Delta}H_-^{-1},
\]
whereas an empty interval contributes zero. The term \(T_{a,a}\) always comes with the correct front factor so that it becomes
\[
 1_{r_i}L_+\overline1^{g(a)}e^{-a\Delta}H_-^{-1}=0.
\]
Hence the singleton and empty
formulas match after projection. This treats the exceptional
times $2\tau=|x_i-x_j|$ and $2\tau=x_i+x_j$ whenever the
collapsed point is positive.

It remains to match across a zero signed endpoint,
$\tau=x_i$. Put
\[
 m=g(0),\qquad T_0=\overline1^mH_-^{-1}.
\]
For $q>0$,
\begin{equation}\label{eqn:411}
       \begin{split}
 W_{0+,q}-W_{0-,q}
 =(P_{0,q}^g-P_{0^+,q}^g)e^{-q\Delta}H_-^{-1}+T_0^*(HT_q-I).
 \end{split}
\end{equation}
The first term is zero when there is no excess wall height.
Otherwise its left factor is supported on heights at most
$m$, by \eqref{eq:endpoint-captures}.

This change can only happen at the minus component of
observation $i$. Allowedness gives $m\le r_i$, and
$\tau=x_i$ gives
\[
 \begin{aligned}
 1_{r_i}S_{-\tau,-x_i}\overline1^m
 &=1_{r_i}L_+\overline1^m=0,\\
 T_0S_{\tau,-x_i}1_{r_i}
 &=\overline1^mL_-H_-^{-1}1_{r_i}=0.
 \end{aligned}
\]
The first identity kills the first term in \eqref{eqn:411}; the second kills the second term in \eqref{eqn:411}.

If the other argument is strictly negative, the mixed
interval is empty near the crossing. If both arguments
vanish, the additional expression is
\[
 T_0-T_0^*+T_0^*HT_0,
\]
and the same projections annihilate all three terms.
An approaching mixed interval is either empty or has limiting
restriction $\overline1^vH_-^{-1}$ with $v\le m$, which is
annihilated in the same way.

These calculations remain valid when wall absorption or the
opening of an initial wall peak occurs at the same time.
Step~2 identifies the upper wall value retained in the
finite-product limit; any omitted part of that restriction
is supported on heights at most $m$ and is removed by the
displayed projections. Thus the two sides agree with the
prescribed right-time value of the kernel at $\tau=x_i$.

\smallskip\noindent
\emph{4. Trace-norm matching.}
On either regular side of $s_0$, Steps~1--3 identify the limit
of every grouped kernel entry with its value at $(\tau,g)$.
If there are simultaneous exceptional cases, expand a difference of finite
products by changing one factor at a time. 

All the entrywise limits have a common summable bound, since the original profile has a finite ceiling on $[0,B]$, and deterministic evolution cannot exceed that ceiling on
any of the dependence intervals under consideration. So Proposition~\ref{prop:trace-class} supplies a uniform
bound for the grouped, projected, weighted entries. Dominated convergence and
\eqref{eq:trace-by-entries} yield
\[
 \left\|
 \mathcal M J
 \bigl[K(\tau_s,g_s)-K(\tau,g)\bigr]\mathcal M
 \right\|_1
 \longrightarrow0.
\]
This proves \eqref{eqn:shiftcontinuity}.

\smallskip\noindent
\emph{5. Right continuity at a fixed profile.}
Fix $t>0$ and keep $h$ fixed. We first establish the
limits of the hitting kernels as the time parameter decreases
from $t+\varepsilon$ to $t$.

For any fixed $a>0$, local finiteness of the profile implies
that $h$ is constant with value $h(a^-)$ immediately to the
left of $a$. Similarly, for $b\ge0$, it is constant with
value $h(b^+)$ immediately to the right of $b$. Hence, for
all sufficiently small $\varepsilon>0$,
\[
 Q_{a-\varepsilon,a}^h
 =Q_\varepsilon^{(h(a^-))}1_{h(a)},
 \qquad
 Q_{b,b+\varepsilon}^h
 =1_{h(b)}Q_\varepsilon^{(h(b^+))}.
\]
Upper semicontinuity gives
$h(a^-)\le h(a)$ and $h(b^+)\le h(b)$. Since \( \|Q_\varepsilon^{(m)}-1_m\|_{\mathrm{op}}
 \le e^{4\varepsilon}-1\), we obtain, in operator norm on $\ell^2(\mathbb Z)$,
\[
 Q_{a-\varepsilon,a}^h\to 1_{h(a)},
 \qquad
 Q_{b,b+\varepsilon}^h\to 1_{h(b)}.
\]
For $0<a\le b$, the constrained Markov property now gives
\[
 \begin{aligned}
 Q_{a-\varepsilon,b+\varepsilon}^h
 &=Q_{a-\varepsilon,a}^h
   Q_{a,b}^hQ_{b,b+\varepsilon}^h\to 
 1_{h(a)}Q_{a,b}^h1_{h(b)}
 =Q_{a,b}^h.
 \end{aligned}
\]
If the
left endpoint is zero, we have the similar convergence. Subtracting from the convergent free heat kernels yields \( P_{a-\varepsilon,b+\varepsilon}^h
 \to P_{a,b}^h,  P_{0,b+\varepsilon}^h
 \to P_{0,b}^h\).
These limits also cover collapsed intervals: \(a=b, b=0\) or \(a>b\).

Suppose first that $t\ne x_j$ for every $j$. For small
$\varepsilon$, the signs of \(W_{t+\epsilon\pm x_i,t+\epsilon\pm x_j}\) are fixed. The preceding calculations
therefore give its limit as the one in $K(t,h)$. If $t=x_j$, $t+\varepsilon-x_j>0$.
$P_{0,\varepsilon}^h\to\overline1^{h(0)}$, and all other
factors have the limits described above. A common upper bound for \(h\) on a
slightly enlarged dependence interval gives the same
summable bound. So using the proof of Proposition~\ref{prop:trace-class} and dominated convergence, we have
\[
 \lim_{\varepsilon\downarrow0}
 \Fhat_n(t+\varepsilon,h)=\Fhat_n(t,h).
\]
Here we do not assume \((t,h)\in \mathcal{R}_{\boldsymbol{x}}\).

Finally, the right-time limit at zero is
Proposition~\ref{thm:initial}. This proves \eqref{eqn:rightcontinuity}.

\end{proof}
\subsection{Verification of the integral Kolmogorov equation}\label{subsec:verification-mild}
We now integrate the regular backward equation, using
Proposition~\ref{thm:transmission} to join its values at exceptional
configurations. Fix $T,B$ as in \eqref{eq:weak-fixed-data}, let $f$
be the observation function \eqref{eq:observation-event}, and put
$G(t,h)=\Fhat_n(t,h)$ for $h\in\cX_B$.
To evaluate $G$, extend $h$ constantly by $h(B)$ beyond $B$.
For $0\le t\le T+1$, all hitting intervals lie inside $[0,B]$,
so the value is independent of this extension.

\begin{proposition}\label{thm:verification-mild}
For $0\le c<1$, the candidate $G$ is a bounded Borel solution of
the integral Kolmogorov equation~\eqref{eq:mild} on
$[0,T]\times\cX_B$, with initial value $f$.
\end{proposition}
\begin{proof}

The finite-product formulas and convergent Fredholm series make $G$
jointly Borel in $(t,h)$.
Propositions~\ref{thm:initial}
and~\ref{thm:uniform} give
\begin{equation}
 G(0,h)=f(h),\qquad |G(t,h)|\le C
       \quad(0\le t\le T+1,\ h\in\cX_B).
\end{equation}
We first want to shift the trajectory so that it has finitely many exceptional times.
Fix $0<t\le T$, $h\in\cX_B$, and $0<\varepsilon<1$. Define
\[ 
 h_s=\phi_s^B h,\qquad \tau_s=t+\varepsilon-s,\qquad
 v_\varepsilon(s)=G(\tau_s,h_s),\qquad 0\le s\le t.
\]
 The
deterministic semigroup property and \eqref{eq:deterministic-envelope} give
\begin{equation}\label{eq:envelope-along-flow}
 r_0(\tau_s,x_j;\phi_s^B h)
 =r_0(t+\varepsilon,x_j;h).
\end{equation}
Hence all \((\tau_s,h_s)\) are allowed, or one fixed threshold is violated
throughout.

We choose $\varepsilon$ so that only finitely many times are
exceptional. Notice the condition that a step meets a hitting endpoint:
 \[
  z_i-\sigma_i s=(t-s)+x_j
  \qquad\text{or}\qquad
  z_i-\sigma_i s=|(t-s)-x_j|.
 \] 
Same sign on \(s\) will make \(s\) to be exception for a whole interval.
There are only finitely many such conditions on the total
time. Replacing $t$ by $t+\varepsilon$ and avoiding these intervals of exceptional time. The remaining exceptional times are finite in number.

Write $\cL_{\rm det}=-\sum_i\sigma_i\partial_{y_i}$.
On an interval where \(v_{\epsilon}(s)\) is differentiable, Theorem~\ref{thm:classical} gives
\[
 v_\varepsilon'(s)
 =-\partial_\tau G(\tau_s,h_s)+\cL_{\rm det}G(\tau_s,h_s)
 =(\cL_{\rm det}-\cL_c)G(\tau_s,h_s).
\]
Expanding the nucleation parts of the generator yields
\begin{equation}\label{eq:shifted-trajectory-ode}
 \begin{split}
 v_\varepsilon'(s)
 &=-\int_{[0,B)}[G(\tau_s,C_zh_s)-G(\tau_s,h_s)]\,\mu_B(dz)\\
 &=\lambda_Bv_\varepsilon(s)
       -\int_{[0,B)}G(\tau_s,C_zh_s)\,\mu_B(dz).
 \end{split}
\end{equation}
Here $\mu_B([0,B))=\lambda_B$, including the wall atom $c\delta_0$.
If \((\tau_s,h_s)\notin \mathcal{A}_{\boldsymbol{r}}\), then also \((\tau_s,C_zh_s)\notin \mathcal{A}_{\boldsymbol{r}}\). Theorem~\ref{thm:FV}
then gives $G(\tau_s,h_s)=G(\tau_s,C_zh_s)=0$, so
\eqref{eq:shifted-trajectory-ode} holds on this curve as well.

On an allowed curve, Proposition~\ref{thm:transmission} matches the
regular pieces with the actual values at their endpoints, including
$s=0,t$. Bound~\eqref{eq:global-candidate-bound} gives
$|v_\varepsilon'|\le2\lambda_BC$ on each piece. Integrating this
bound, extending to endpoints by matching, and adding over the finite
partition show that $v_\varepsilon$ is Lipschitz on $[0,t]$.
On a forbidden curve it is identically zero. In both cases it is
absolutely continuous and satisfies \eqref{eq:shifted-trajectory-ode}
almost everywhere.

Multiplying \eqref{eq:shifted-trajectory-ode} by $e^{-\lambda_Bs}$ and integrating over \([0,t]\), we get
\begin{equation}\label{eq:shifted-integral-identity}
 \begin{split}
 G(t+\varepsilon,h)=e^{-\lambda_Bt}G(\varepsilon,\phi_t^B h)
 &+\int_0^t e^{-\lambda_Bs}
      \int_{[0,B)}G(t+\varepsilon-s,C_z\phi_s^B h)\,\mu_B(dz)\,ds.
 \end{split}
\end{equation}
Sending $\varepsilon\downarrow0$, Proposition~\ref{thm:transmission} and dominated convergence give \eqref{eq:mild}.
\end{proof}

\subsection{Proof of the main theorem}\label{subsec:continuation}
Proposition~\ref{thm:verification-mild} identifies the Pfaffian with the
transition probability for $0\le c<1$. Analytic continuation then
extends the identity to all nonnegative boundary rates.
\begin{proof}[Proof of Theorem~\ref{thm:multipoint}]
Fix $t>0$, a finite-valued initial profile $h$, and \(\boldsymbol{x},\boldsymbol{r}\). Let $p_k$
be the probability of the height event with the usual bulk marks and
exactly $k$ independent uniform wall times in $[0,t]$, ordered in
time. These probabilities are independent of $c$ and lie in $[0,1]$.
Conditioning on the Poisson number of wall marks gives, for real $c\ge0$,
\[
 \mathbb P_h^c\bigl(h(t,x_j)\le r_j,\ 1\le j\le n\bigr)
       =e^{-ct}\sum_{k=0}^{\infty}\frac{(ct)^k}{k!}\,p_k.
\]
For complex $c$, the sum of absolute values is at most
$e^{t(|c|-\operatorname{Re}c)}$, uniformly on compact sets.
The right-hand side is therefore entire. 

The scalar Pfaffian is also entire in $c$. To see this,
fix $C_0<\infty$ and choose the weights in \eqref{eq:weights}
once for the disk $|c|\le C_0$. Write \(A(c):=\mathcal M J K(t,h;c)\mathcal M\).
The entries of $H_c$ and $H_-^{-1}$ are polynomials in $c$,
by \eqref{eq:H-kernel} and \eqref{eq:H-inverse-entries}, and the rest parts of kernel are independent of \(c\). The estimates in the proof of
Proposition~\ref{prop:trace-class} apply equally to complex
$c$ with $|c|\le C_0$, so every
entry of $A(c)$ is entire and the trace norm is uniformly bounded in \(c\). Thus, $\Pf(J+A(c))$ is entire. 

By Proposition~\ref{thm:verification-mild}, the two entire
functions agree for $0\le c<1$, so it gives \eqref{eq:multi-result} for every $c\ge0$ and $t>0$.

We now extend the identity to every $h\in\cX$, allowing
the value $-\infty$. Fix $c\ge0$ and put
\[
 h^{(M)}:=h\vee(-M),\qquad M\in\mathbb N.
\]
These profiles are finite-valued and decrease to $h$.
Couple their evolutions with that from $h$ using the same
graphical marks, with the convention $-\infty+1=-\infty$.
Choose $B>x_n+t+1$, and let $N_B$ be the total number of
bulk and wall marks in $[0,t]\times[0,B]$. So,
\[
 h(t,x_j;h)
 \le h(t,x_j;h^{(M)})
 \le \max\{h(t,x_j;h),-M+N_B\}.
\]
Since $N_B<\infty$ almost surely, the height-event
indicators for $h^{(M)}$ and $h$ agree for all sufficiently
large $M$. Their probabilities therefore converge.

On the kernel side, for each $0\le a\le b\le B$ and
$u,v\in\mathbb Z$,
\[
 P_{a,b}^{h^{(M)}}(u,v)\downarrow P_{a,b}^h(u,v).
\]
 All the truncated profiles
have the common finite ceiling.
The same estimates in the proof of
Proposition~\ref{prop:trace-class} provide a common bound for the weighted kernel entries.
Dominated convergence and \eqref{eq:trace-by-entries}
therefore give
\[
 \left\|
 \mathcal M J
 \bigl(K(t,h^{(M)})-K(t,h)\bigr)\mathcal M
 \right\|_1\longrightarrow0.
\]

Finally, we extend it to \(t=0\). fix a real $c\ge0$ and take $B>x_n+1$. For $0<t<1$,
finite propagation and the event that there are no
nucleations in the dependence box give
\[
 \left|
 \mathbb P_h^c\bigl(h(t,x_j)\le r_j,\ 1\le j\le n\bigr)
 -f(\phi_t h)
 \right|
 \le 1-e^{-(2B+c)t}.
\]
Upper semicontinuity and the discreteness of the heights
imply that $f(\phi_t h)=f(h)$ for all sufficiently small $t$.
The transition probability therefore tends to $f(h)$ as
$t\downarrow0$. Taking this limit in the positive-time
identity proves that the prescribed right-time limiting
Pfaffian at $t=0$ equals $f(h)$, completing the proof.

\end{proof}

\appendix
\section{Analytic properties of the Fredholm Pfaffian}\label{sec:analytic}
In the Appendix, we will prove proposition \eqref{prop:trace-class}. We will use the following bound:
\begin{equation}\label{eq:trace-by-entries}
 \|A\|_1\le\sum_{i,j=1}^n\sum_{\alpha,\beta\in\{-,+\}}
                    \sum_{u,v\in\mathbb Z}|A_{ij}^{\alpha\beta}(u,v)|.
\end{equation}

Fix a time bound $T$, a spatial bound $B$, and assume $|c|\le C_0$. Let $h_* = \max\{h(y):y\in[0,B]\}$. Constants below may depend on
these bounds, but not on the number or
positions of the steps. Write $r_+=\max(r,0)$ for the positive part.

\begin{lemma}\label{lem:tails}
For every $\lambda>0$ and $T,B<\infty$, there is a finite constant
such that each kernel
\[
 D\in\{S_{t,x},\ e^{x\Delta},\ L_+(t),\ L_-(t)\},
 \qquad |t|\le T,\quad |x|\le B,
\]
satisfies
\begin{equation}\label{eq:free-tail}
 \sup_u\sum_v|D(u,v)|e^{\lambda|u-v|}
 +\sup_v\sum_u|D(u,v)|e^{\lambda|u-v|}<\infty,
\end{equation}
uniformly over \(t,x\). The factors are defined in
\eqref{eq:free} and \eqref{eq:Lpm}.

For $0\le a\le b\le B$,
\begin{align}
 P_{a,b}^h(u,v)&\le C_{\lambda,B}
 e^{-\lambda((u-h_*)_++(v-h_*)_++|u-v|)},\label{eq:hit-tail}\\
 \sum_{w\ge v}P_{a,b}^h(u,w)&\le C_{\lambda,B}
 e^{-\lambda((u-h_*)_++(v-h_*)_++(v-u)_+)},
 \label{eq:hit-cumulative}
\end{align}
with $C_{\lambda,B}=\exp\{2B(e^{2\lambda}-1)\}$.
The bounds include $a=b$ under the closed-interval convention.
\end{lemma}
\begin{proof}
      For \(a\in \mathbb{R}\), \(e^{a\tau_{\pm}}(u,u\pm k) = \frac{a^k}{k!}, k>0\). Thus,
      \[\sum_{v}|e^{a\tau_{\pm}}(u,v)|e^{\lambda|u-v|}= \exp\{|a|e^{\lambda}\}.\]
We write \(S_{t,x}, e^{x\Delta}, L_{\pm}\) in terms of \(\tau_{\pm}\):
\[
 \begin{split}
 S_{t,x}=e^{-2x}e^{(x+t)\tau_+}e^{(x-t)\tau_-},\qquad e^{x\Delta}&=e^{-2x}e^{x\tau_+}e^{x\tau_-},\qquad
 L_\pm(t)=e^{2t}e^{-2t\tau_\pm}.
 \end{split}
\]
Thus, \(\sup_u\sum_v|D(u,v)|e^{\lambda|u-v|}\) is bounded respectively by
\[
 e^{2|x|+(|x+t|+|x-t|)e^\lambda},\qquad
 e^{2|x|+2|x|e^\lambda},\qquad
 e^{2|t|+2|t|e^\lambda}.
\]
Translation invariance gives the bounds on \(\sup_v\sum_u|D(u,v)|e^{\lambda|u-v|}\).

Let $N$ be the number of walk jumps on $[a,b]$. It is Poisson with
mean $2(b-a)\le2B$. A path from $u$ to $v$ that hits a profile
bounded above by $h_*$ must satisfy
\[
 N\ge(u-h_*)_++(v-h_*)_+,\qquad N\ge|u-v|.
\]
Let \(E = (u-h_*)_++(v-h_*)_++|u-v|\). Exponential Markov's inequality gives
\[P^{h}_{a,b}(u,v)\leq e^{-\lambda  E}\mathbb{E}(e^{2\lambda N})\leq e^{-\lambda E}\exp\{2B(e^{2\lambda}-1)\}.\]
For a terminal height at least $v$, replace the second lower
bound by $N\ge(v-u)_+$, and the same arguments hold.
\end{proof}

\paragraph{The growth of the boundary factor.}
The explicit kernel \eqref{eq:H-kernel} gives, for $0\leq c\le C_0$,
\begin{equation}\label{eq:H-growth}
 |H_c(u,v)|\le\tfrac12 e^{\omega|u-v|}.
\end{equation}
By \eqref{eq:free-tail}, the sums defining $H_cS_{t,x}$ and
$e^{x\Delta}H_ce^{y\Delta}$ converge absolutely and locally uniformly
in $c$. They satisfy
\begin{equation}\label{eq:H-free-growth}
 |(H_cS_{t,x})(u,v)|+
 |(e^{x\Delta}H_ce^{y\Delta})(u,v)|
 \le C e^{\omega|u-v|}
\end{equation}
for bounded $t,x,y$.

The one-sided kernel $H_-^{-1}$ has the explicit kernel
\begin{equation}\label{eq:H-inverse-entries}
 H_-^{-1}(u,v)=
 \begin{cases}
 0,&u<v,\\
 -2c,&u=v,\\
 2(1+c^2),&u-v>0\text{ odd},\\
 -4c,&u-v>0\text{ even}.
 \end{cases}
\end{equation}
This follows by multiplying the finite-shift factor in \eqref{eq:A}
into its one-sided series. In particular these entries are bounded
uniformly on $0\leq c\le C_0$, and the kernel is lower triangular.
The same holds for $L_-H_-^{-1}$: its entries are bounded by the
absolute row sum of $L_-$ times the bound in
\eqref{eq:H-inverse-entries}.

\begin{proof}[Proof of Proposition~\ref{prop:trace-class}]
Recall the definition of \(G_{i,\pm}\) in \eqref{eq:G-exact}, \(L_{\pm}\) in \eqref{eq:Lpm}. We first establish the following two bounds: for every $\lambda>0$, uniformly on $|c|\le C_0$,
\begin{align}
 |G_{j,-}(z,v)|+|G_{j,+}(z,v)|&\le C_\lambda e^{-\lambda|z|-\lambda v},
 &&v>r_j,\label{eq:G-tail}\\
 |(L_+P_{a,b}^hL_-)(u,v)|+|(L_+P_{a,b}^hL_-H_-^{-1})(u,v)|
 &\le C_\lambda e^{-\lambda(u+v)},
 &&u>r_i,\ v>r_j,\label{eq:direct-tail}
\end{align}

Since \(L_-, L_-H^{-1}_-\) are uniformly bounded, we have \(|G(z,v)|\leq C \sum_{w\geq v}P_{0,s}^{h}(z,w)\). The exponent in
\eqref{eq:hit-cumulative} is at least $|z|+v-C'$:
for $z\ge0$, \((z-h_*)_++(v-h_*)_+>|z|+v-2h_*^+\); for $z<0$, $(v-z)_+\ge v+|z|$. 

For \eqref{eq:direct-tail}, the support and boundedness of the outer factors imply that it is bounded by $C\sum_{z\ge u,w\ge v}P_{a,b}^h(z,w)$.
Apply \eqref{eq:hit-tail} with exponent $2\lambda$,
\[\sum_{z\geq u,w\geq v}P_{a,b}^{h}(z,w)\leq C_\lambda e^{2\lambda h_*}(\sum_{z\geq u}e^{-\lambda z})(\sum_{w\geq v}e^{-\lambda w})\leq C_{\lambda,h_*} e^{-\lambda(u+v)}.\]

Now we bound the kernels in \eqref{eq:FG-terms} and \eqref{eq:positive-grouped}. There are three types of terms in \eqref{eq:FG-terms}: \(S G, (HS)^*G\) and \(G^*HG\). Choose $\lambda>\omega$, \eqref{eq:G-tail} gives
\[
 \begin{split}
      \sum_{z,w}|G_{i,\alpha}(z,u)H_c(z,w)G_{j,\beta}(w,v)|&\le C_\lambda e^{-\lambda(u+v)}\sum_{z,w}e^{-\lambda|z|}e^{\omega |z-w|}e^{-\lambda |w|}\\
 &\le C_\lambda e^{-\lambda(u+v)}
       \left(\sum_z e^{-(\lambda-\omega)|z|}\right)^2.
 \end{split}
\]
Using \eqref{eq:G-tail} and  \eqref{eq:free-tail},
\[|(S_{-t,-x_i}G_{j,\beta})(u,v)|\leq C_\lambda e^{-\lambda v}\sum_{z}|S(u,z)|e^{-\lambda|z|}\leq C_\lambda e^{-\lambda|u|-\lambda v}\le  C_\lambda e^{-\lambda(u+v)}.\]
Using \eqref{eq:G-tail} and \eqref{eq:H-free-growth},
\[ |((HS_{t,x_i})^*G_{j,\beta})(u,v)|
 \le C_\lambda e^{-\lambda v}
      \sum_z e^{\omega|z-u|-\lambda|z|}
 \le C_\lambda e^{\omega u-\lambda v}.\]
In the last inequality, $u>r_i$ permits replacing $|u|$ by $u$.
Taking adjoint handles the other terms. Hence, we have
\begin{equation}\label{eq:scat-tail}
 |(J\widetilde K)_{ij}(u,v)|\le C_\lambda
 \begin{pmatrix}
 e^{-\lambda(u+v)}&e^{-\lambda u+\omega v}\\
 e^{\omega u-\lambda v}&e^{\omega u-\lambda v}+e^{-\lambda u+\omega v}
 \end{pmatrix},\qquad u>r_i,\ v>r_j.
\end{equation}

Choose $\lambda>\eta_n>\omega$ in \eqref{eq:scat-tail}. After the conjugation, \(\mathcal{M}(J\tilde{K})\mathcal{M}\) has all the exponents negative. 

It remains to consider the terms in \(R\). The $(+,+)$ component of $\mathcal{M}JR\mathcal{M}$ is 
\[
 e^{-\eta_i u-\eta_j v}|(e^{x_i\Delta}H_ce^{x_j\Delta})(u,v)|\leq Ce^{-\eta_i u-\eta_j v+\omega |u-v|}.
\]
Since \(u>r_i,v>r_j\), it is summable. For the free heat term with $i<j$, write
$q_d=e^{(x_j-x_i)\Delta}(0,d)$. Then
\begin{equation}\label{eq:free-heat-sum}
\begin{aligned}
 \sum_{u>r_i,v>r_j}e^{\eta_i u-\eta_jv}|q_{v-u}|
 &=\sum_d|q_d|e^{-\eta_jd}
     \sum_{u>\max(r_i,r_j-d)}e^{-(\eta_j-\eta_i)u}\\
 &\le C\sum_d|q_d|e^{\eta_j|d|}<\infty.
\end{aligned}
\end{equation}
This is where the strict ordering of
$\eta_i$ is used.

The bound~\eqref{eq:trace-by-entries} now gives trace class with the
stated uniformity.
\end{proof}
In the following corollary, we establish that all the derivative, integral and limit hold in the trace norm.
\begin{corollary}\label{cor:everythingtrueintracenorm}
For \((t,h)\in \mathcal{R}_{\boldsymbol{x}}\), let \(A(t,h):=\mathcal{M}JK(t,h)\mathcal{M}\).
\begin{enumerate}
      \item \(A(t,h)\) is differentiable in the trace norm.
      \item \(A(t,C_zh)-A(t,h)\) is integrable in trace norm over interval \([0,\infty)\).
\end{enumerate}

\end{corollary}
\begin{proof}
Fix the weights $\mathcal M$ and choose $B>x_n+t$. By regularity, there is a neighborhood $\mathcal U$ in the
time and step coordinates on which function heights on each constant interval,
the ordering of steps, hitting endpoints, and the regimes
in \eqref{eq:Wpq} remain unchanged. All profiles in this
neighborhood have a common finite ceiling $H$ on $[0,B]$.

\emph{Differentiability.}
Notice that the derivatives
are sums of translates of $P$ (terms from in \eqref{eq:endpoint-single}) and no-hit products forced
through neighboring heights at most $H+1$ (terms from \eqref{eqn:deterministicshift}). The jump-count argument of Lemma~\ref{lem:tails} therefore gives
\eqref{eq:hit-tail} and \eqref{eq:hit-cumulative} also for
the absolute values of these derivatives, with an enlarged
ceiling and constants, uniformly on $\mathcal U$.
Derivatives of \(S_{\pm t,\pm x}\) are finite sums of
translates and satisfy \eqref{eq:free-tail}; derivatives
of $L_\pm$ also retain their triangular support.

Apply the product rule to the grouped expressions
\eqref{eq:FG-terms}--\eqref{eq:positive-grouped}.
Using the same argument as in the proof of
Proposition~\ref{prop:trace-class}, we have
\[
 \sum_{i,j=1}^n\sum_{\alpha,\beta\in\{-,+\}}
 \sum_{u,v\in\mathbb Z}
 \sup_{\mathcal U}
 \bigl|D A_{ij}^{\alpha\beta}(u,v)\bigr|<\infty.
\]
The entrywise mean-value formula supplies the same summable
majorant for difference quotients. Dominated convergence,
together with \eqref{eq:trace-by-entries}, proves
differentiability in trace norm.

\emph{Integrability of nucleation increments.}
Put $\delta_zA=A(t,C_zh)-A(t,h)$.
Every hitting interval lies in $[0,x_n+t]$, so
$\delta_zA=0$ for $z>x_n+t$.
For $0\le z\le x_n+t$, the profiles $C_zh$ have the common
ceiling $H+1$. The estimates of
Proposition~\ref{prop:trace-class}, which do not depend on
the number or positions of the steps, consequently give
a common summable majorant for the entries of $\delta_zA$.

The finite-product representation \eqref{eq:plateau-product}, with the additional
restriction at $z$, shows that these entries are piecewise
continuous in $z$, with possible exceptions only at the
finitely many steps and hitting endpoints. We also have the bound
\[
 \int_0^\infty\|\delta_zA\|_1\,dz
 \le (x_n+t)\sup_{0\le z\le x_n+t}\|\delta_zA\|_1
 <\infty.
\]
Thus the increment is integrable in trace class.
\end{proof}
\begin{lemma}\label{lem:pfcalculus}
Let $A$ be a skew trace-class kernel on $\HH$. Let \(\Phi(A):=\Pf(J+A)\). Suppose $A=A(s)$ is differentiable in trace norm and \(I-JA(s)\) is invertible. Then
\begin{equation}\label{eq:pf-derivative}
 \frac{d}{ds}\Phi(A(s))
 =
 -\frac12\Phi(A(s))
 \tr\!\left((I-JA(s))^{-1}J A'(s)\right).
\end{equation}
If $B$ is a skew rank-two operator and  \(I-JA(s)\) is invertible,
then
\begin{equation}\label{eq:pf-affine}
 \Phi(A+B)-\Phi(A)
 =
 -\frac12\Phi(A)\tr\!\left((I-JA)^{-1}JB\right).
\end{equation}
If \(I-JA\) is not invertible, it is still true that
\[A'(s) =B \Rightarrow \frac{d}{ds}\Phi(A(s))= \Phi(A+B)-\Phi(A).\]
\end{lemma}

\begin{proof}
If $I-JA(s)$ is invertible, \eqref{eq:pf-derivative} follows from the fact that \(\Pf^2(J+K)=\det(I-JK)\) and 
\[
 \frac{d}{ds}\det(I-JA(s))
 =
 -\det(I-JA(s))
 \tr\!\left((I-JA(s))^{-1}J A'(s)\right).
\]

To prove \eqref{eq:pf-affine}, assume \(B\) has the form \( B:=u\otimes v-v\otimes u\). Define the maps $U:\mathbb C^2\to\HH$ and $V:\HH\to\mathbb C^2$
by
\[
 U(a,b)=au+bv,\qquad
 Vf=\begin{pmatrix}v^{\mathsf T}f\\-u^{\mathsf T}f\end{pmatrix}.
\]
Thus, $B=UV$.

Since $I-JA$ is invertible, let
\[
 R=(J+A)^{-1}=-(I-JA)^{-1}J.
\]
Since $J+A$ is skew, $R^{\mathsf T}=-R$. Consequently,
\[
 u^{\mathsf T}Ru=v^{\mathsf T}Rv=0,
 \qquad
 u^{\mathsf T}Rv=-v^{\mathsf T}Ru.
\]
Writing $\alpha=v^{\mathsf T}Ru$, the determinant-square
identity and the finite-rank determinant identity give
\[
 \begin{aligned}
 \Phi(A+B)^2
 &=\det(I-JA-JB)=\Phi(A)^2\det(I+RB)\\
 &=\Phi(A)^2\det(I+VRU)= \Phi(A)^2\det\begin{pmatrix}
  1+v^{\mathsf T}Ru & v^{\mathsf T}Rv\\
  -u^{\mathsf T}Ru & 1-u^{\mathsf T}Rv
 \end{pmatrix}\\
 &=
 \Phi(A)^2\det\begin{pmatrix}
  1+\alpha&0\\
  0&1+\alpha
 \end{pmatrix}
 =\Phi(A)^2(1+\alpha)^2.
 \end{aligned}
\]
So, \(\Phi(A+B)=\Phi(A)(1+\alpha)\). On the other hand, \(
 \tr(RB)
 =v^{\mathsf T}Ru-u^{\mathsf T}Rv
 =2\alpha\), so
\[
 \alpha
 =-\frac12\tr\!\left((I-JA)^{-1}JB\right).
\]
This proves \eqref{eq:pf-affine}.

When \(I-JA\) is not known to be invertible a prior, we use the Fr\'echet derivative; see \cite[Lemma 4.11]{MQR}.
\end{proof}

\bigskip
\begingroup\small
\noindent\textsc{Kailun Chen}\\
Department of Mathematics, National University of Singapore,\\
2 Science Drive 2, Singapore 117543\\
Email address: \href{mailto:kailunxsx@gmail.com}{\nolinkurl{kailunxsx@gmail.com}}\\
URL: \url{https://sites.google.com/view/kailunchenshomepage/home}

\medskip
\noindent\textsc{Xincheng Zhang}\\
Department of Mathematics, California Institute of Technology,\\
Pasadena, California, USA\\
Email address: \href{mailto:xczhang@caltech.edu}{\nolinkurl{xczhang@caltech.edu}}\\
URL: \url{https://xinchengzhang96.github.io/}
\par\endgroup

\begin{thebibliography}{99}
\raggedright

\bibitem{ABW21}
A.~Aggarwal, A.~Borodin, and M.~Wheeler.
Deformed polynuclear growth in $(1+1)$ dimensions.
\emph{Int. Math. Res. Not. IMRN} \textbf{2023}, no.~7,
5728--5780.

\bibitem{BR00}
J.~Baik and E.~M.~Rains.
Limiting distributions for a polynuclear growth model with external sources.
\emph{J. Stat. Phys.} \textbf{100} (2000), 523--541.

\bibitem{BFS08}
A.~Borodin, P.~L.~Ferrari, and T.~Sasamoto.
Large time asymptotics of growth models on space-like paths II: PNG and parallel TASEP.
\emph{Comm. Math. Phys.} \textbf{283} (2008), 417--449.

\bibitem{Fer04}
P.~L.~Ferrari.
Polynuclear growth on a flat substrate and edge scaling of GOE eigenvalues.
\emph{Comm. Math. Phys.} \textbf{252} (2004), 77--109.

\bibitem{IS04}
T.~Imamura and T.~Sasamoto.
Fluctuations of the one-dimensional polynuclear growth model with external sources.
\emph{Nuclear Phys. B} \textbf{699} (2004), 503--544.

\bibitem{Joh03}
K.~Johansson.
Discrete polynuclear growth and determinantal processes.
\emph{Comm. Math. Phys.} \textbf{242} (2003), 277--329.

\bibitem{JR21}
K.~Johansson and M.~Rahman.
Multitime distribution in discrete polynuclear growth.
\emph{Comm. Pure Appl. Math.} \textbf{74} (2021), 2561--2627.

\bibitem{KPZ86}
M.~Kardar, G.~Parisi, and Y.-C.~Zhang.
Dynamic scaling of growing interfaces.
\emph{Phys. Rev. Lett.} \textbf{56} (1986), 889--892.

\bibitem{MQR}
K.~Matetski, J.~Quastel, and D.~Remenik.
Polynuclear growth and the Toda lattice.
\emph{J. Eur. Math. Soc.} \textbf{28} (2026), no.~11,
4951--5021.

\bibitem{PS00}
M.~Pr\"ahofer and H.~Spohn.
Universal distributions for growth processes in $1+1$ dimensions and random matrices.
\emph{Phys. Rev. Lett.} \textbf{84} (2000), 4882--4885.

\bibitem{PS02}
M.~Pr\"ahofer and H.~Spohn.
Scale invariance of the PNG droplet and the Airy process.
\emph{J. Stat. Phys.} \textbf{108} (2002), 1071--1106.

\bibitem{PS04}
M.~Pr\"ahofer and H.~Spohn.
Exact scaling functions for one-dimensional stationary KPZ growth.
\emph{J. Stat. Phys.} \textbf{115} (2004), 255--279.

\bibitem{Rains}
E.~M.~Rains.
\emph{Correlation functions for symmetrized increasing subsequences}.
Preprint, arXiv:math/0006097, 2000.

\bibitem{SI04}
T.~Sasamoto and T.~Imamura.
Fluctuations of the one-dimensional polynuclear growth model in half-space.
\emph{J. Stat. Phys.} \textbf{115} (2004), 749--803.
arXiv:cond-mat/0307011v1.

\bibitem{Zhang24}
X.~Zhang.
TASEP in Half-space. 
Preprint, arxiv.org/abs/2409.09974.

\end{thebibliography}
\end{document}